\documentclass[11pt]{amsart}
\usepackage{graphicx, color, epsfig}
\usepackage{amscd}
\usepackage{amsmath,amsthm,empheq}
\usepackage{amsfonts}
\usepackage{amssymb}
\usepackage{mathrsfs}
\usepackage{graphicx}
\usepackage[all]{xy}

\numberwithin{equation}{section}

\usepackage[latin1]{inputenc}

\usepackage[english]{babel}

\newtheorem{corollary}{Corollary}[section]

\newtheorem{example}[corollary]{Example}

\newtheorem{prop}[corollary]{Proposition}
\newtheorem{rem}[corollary]{Remark}
\newtheorem{thm}[corollary]{Theorem}
\newtheorem{cor}[corollary]{Corollary}

\newfont{\sBlackboard}{msbm10 scaled 900}

\newcommand{\mylabel}[1]{\label{#1}
            \ifx\undefined\stillediting
            \else \fbox{$#1$}\fi }
\newcommand{\BE}{\begin{equation}}

\newcommand{\EEQ}{\end{equation}}
\newcommand{\rfb}[1]{\mbox{\rm
   (\ref{#1})}\ifx\undefined\stillediting\else:\fbox{$#1$}\fi}

\newfont{\Blackboard}{msbm10 scaled 1200}

\newfont{\roma}{cmr10 scaled 1200}
\def\RR{{\mathbb R} }

\newcommand{\NN}{\mathbb{N}}

\newcommand{\bb}{\begin{equation}}
\newcommand{\bbb}{\end{equation}}

\DeclareMathOperator{\divv}{div}

\newcommand{\mm}    {{\hbox{\hskip 0.5pt}}}

\newcommand{\bluff} {{\hbox{\raise 15pt \hbox{\mm}}}}
scaled\magstep2

\makeatletter
\def\section{\@startsection {section}{1}{\z@}{-3.5ex plus -1ex minus
    -.2ex}{2.3ex plus .2ex}{\large\bf}}
\makeatother

\begin{document}

\title[Global multiplicity for parametric anisotropic Neumann $(p,q)$-equations]
{Global multiplicity for parametric anisotropic Neumann $(p,q)$-equations}
\author[Nikolaos S. Papageorgiou - Vicen\c tiu D. R\u adulescu - Du\v san D. Repov\v s]{Nikolaos S. Papageorgiou - Vicen\c tiu D. R\u adulescu -\\ Du\v san D. Repov\v s}
\address
{\textsc{Nikolaos S. Papageorgiou}\\
Department of Mathematics\\
National Technical University\\
Zografou  Campus, Athens 15780\\
GREECE.}
\email{npapg@math.ntua.gr}
\address{\textsc{Vicen\c tiu D. R\u adulescu (Corresponding author)}\\
Faculty of Applied Mathematics\\
AGH University of Science and Technology\\
al. Mickiewicza 30, 30-059 Krakow\\
POLAND\\
\&\\
Department of Mathematics\\
University of Craiova\\
200585 Craiova\\
ROMANIA\\
\&\\
China-Romania Research Center in Applied Mathematics}
\email{radulescu@inf.ucv.ro}
\address
{\textsc{Du\v san D. Repov\v s}\\
Faculty of Education and Faculty of Mathematics and Physics\\
University of Ljubljana\\
1000 Ljubljana\\
SLOVENIA\\
\&\\
Institute for Mathematics, Physics and Mechanics\\
1000 Ljubljana\\
SLOVENIA}
\email{dusan.repovs@guest.arnes.si}
\subjclass[2020] {Primary: 35A16, 35J20; Secondary: 03H05, 35J70, 47J30, 58E05}
\keywords{anisotropic operator; superlinear reaction; positive and nodal solutions; critical groups}
 
\begin{abstract}%
	We consider a Neumann boundary value problem driven by the anisotropic $(p,q)$-Laplacian plus a parametric potential term. The reaction is ``superlinear". We prove a global (with respect to the parameter) multiplicity result for positive solutions. Also, we show the existence of a minimal positive solution and finally, we produce a nodal solution.
\end{abstract}
\maketitle

\section{Introduction} 		      
Let $\Omega\subseteq\RR^N$ be a bounded domain with $C^2$-boundary $\partial\Omega$. In this paper, we study the following parametric anisotropic Neumann $(p,q)$-equation:
\begin{equation}\tag{\mbox{$P_\lambda$}}
\left\{
\begin{array}{lll}
-\Delta_{p(z)} u-\Delta_{q(z)} u+\lambda |u|^{p(z)-2}u= f(z,u) \text{ in } \Omega,\\
\displaystyle\frac{\partial u}{\partial n}=0 \mbox{ on }\partial\Omega,\ \lambda\in\RR.
\end{array}
\right.
\end{equation}

Given $r\in C(\overline{\Omega})$ with $1<\displaystyle{\min_{\overline{\Omega}} r}$, we denote by $\Delta_{r(z)}$  the anisotropic $r$-Laplace differential operator defined by
$$
\Delta_{r(z)}u=\divv\left(|Du|^{r(z)-2}Du\right) \mbox{ for all }u\in W^{1,r(z)}(\Omega).
$$

In contrast to the isotropic $r$-Laplacian (that is, if $r(\cdot)$ is constant), the anisotropic one is not homogeneous and this is a source of difficulties in the study of anisotropic problems. Equation $(P_\lambda)$ is driven by the sum of two such operators. So, even when the exponents are constant functions (isotropic operators), the differential operator of the problem is not homogeneous. There is also a parametric potential term $u\mapsto\lambda|u|^{p(z)-2}u$ with $\lambda\in\RR$ being the parameter. Note that $\lambda$ need not be positive and so the differential operator is not in general coercive. In the reaction (right-hand side of problem $(P_\lambda)$), we have a Carath\' eodory function $f(z,x)$ (that is, for all $x\in \RR$ the mapping $z\mapsto f(z,x)$ is measurable and for a.a. $z\in\Omega$, the function $x\mapsto f(z,x)$ is continuous), which is $(p_+-1)$-superlinear as $x\to\pm\infty$ (here, $p_+=\displaystyle{\max_{\overline{\Omega}}p}$ for $p\in C^{0,1}(\overline{\Omega})$). However, we do not use the Ambrosetti-Rabinowitz condition (the $AR$-condition for short), which is common in the literature when dealing with ``superlinear" problems. Our condition on $f(z,\cdot)$ is less restrictive and incorporates in our framework, also superlinear nonlinearities with ``slower" growth near $\pm\infty$.

Our aim is to study the changes in the set of positive solutions as the parameter $\lambda\in\RR$ moves on the real line. We prove a global multiplicity result (a bifurcation-type result for large values of the parameter). More precisely, we show the existence of a critical parameter value $\lambda_*>-\infty$ such that
\begin{itemize}
  \item for all $\lambda>\lambda_*$, problem $(P_\lambda)$ has at least two positive smooth solutions;
  \item for $\lambda=\lambda_*$, problem $(P_\lambda)$ has at least one positive smooth solution;
  \item for $\lambda\in(0,\lambda_*)$, problem $(P_\lambda)$ has no positive solution.
\end{itemize}

We also establish that for all $\lambda\in [\lambda_*,\infty)$, problem $(P_\lambda)$ has a smallest positive solution. Finally, the extremal constant sign solutions are used to produce a nodal (sign-changing) solution.

Analogous bifurcation type results describing the changes in the set of positive solutions for anisotropic Neumann problems were proved by Fan \& Deng \cite{4Fan-Deng} and Deng \& Wang \cite{1Den-Wan}. They consider problems driven by the $p(z)$-Laplacian and impose restrictive positivity and monotonicity conditions on the reaction $f(z,\cdot)$ and, in addition, they employ the $AR$-condition to express the superlinearity of the reaction. We also mention the recent isotropic work of Papageorgiou \& Zhang \cite{12Pap-Zhang} with a parametric boundary condition.
Finally, further existence and multiplicity results can be found in \cite{refe1, refe2, refe3, refe4} and the references therein.

\section{Mathematical background and hypotheses}

The analysis of problem $(P_\lambda)$ requires the use of function spaces with variable exponents. A comprehensive presentation of the theory of these spaces can be found in the book of Diening, Harjulehto, H\"asto \& Ruzi\v cka \cite{2Din-Har-Has-Ruz}. We also refer to the monograph of R\u adulescu and Repov\v s \cite{13Rad-Rep} for the basic
variational and topological methods used in the treatment of problems with variable exponent.

Let $M(\Omega)$ be the vector space of all measurable functions $u:\Omega\to \RR$. As usual, we identify two such functions which differ only on a Lebesgue-null subset of $\Omega$. Given $r\in C(\overline{\Omega})$, we define
$$
r_-=\min_{\overline{\Omega}}r \mbox{ and } r_+=\max_{\overline{\Omega}}r.
$$
Consider the set
$$
E_1=\{r\in C(\overline{\Omega}):\:1<r_-\}.
$$
Then for $r\in E_1$ we define the variable exponent Lebesgue space $L^{r(z)}(\Omega)$ as follows
$$
L^{r(z)}(\Omega)=\{u\in M(\Omega):\:\rho_r(u)<\infty\},
$$
where $\rho_r(\cdot)$ is the modular function defined by
$$
\rho_r(u)=\int_\Omega |u|^{r(z)}dz.
$$
We equip the space $L^{r(z)}(\Omega)$ with the so called ``Luxemburg norm" defined by
$$
\|u\|_{r(z)}=\inf\left\{\vartheta:\:\rho_r\left(\frac{u}{\vartheta}\right)
\leq1\right\}.
$$
Then $L^{r(z)}(\Omega)$ becomes a Banach space which is separable and reflexive (in fact, uniformly convex). For $r\in E_1$, we define the conjugate variable exponent $r'(\cdot)$ corresponding to $r(\cdot)$ by
$$
r'(z)=\frac{r(z)}{r(z)-1} \mbox{ for all }z\in\overline{\Omega}.
$$
Evidently, $r'\in E_1$ and $\frac{1}{r(z)}+\frac{1}{r'(z)}=1$ for all $z\in\overline{\Omega}$. We know that $L^{r(z)}(\Omega)^*=L^{r'(z)}(\Omega)$ and the following H\"older's inequality is true
$$
\int_\Omega |uv|dz\leq \left(\frac{1}{r_-}+\frac{1}{r'_-}\right)\|u\|_{r(z)}\|v\|_{r'(z)}
$$
for all $u\in L^{r(z)}(\Omega)$, all $v\in L^{r'(z)}(\Omega)$.

Having the variable exponent Lebesgue spaces, we can define the corresponding variable exponent Sobolev spaces. So, if $r\in E_1$, then we define
$$
W^{1,r(z)}(\Omega)=\left\{u\in L^{r(z)}(\Omega):\: |Du|\in L^{r(z)}(\Omega)\right\},
$$
with $Du$ being the weak gradient of $u(\cdot)$. We equip this space with the following norm
$$
\|u\|_{1,r(z)}=\|u\|_{r(z)}+\|Du\|_{r(z)} \mbox{ for all }u\in W^{1,r(z)}(\Omega)
$$
with $\|Du\|_{r(z)}= \| |Du| \|_{r(z)}$. It follows that $W^{1,r(z)}(\Omega)$ is a Banach space which is separable and reflexive (in fact, uniformly convex).

For $r\in E_1$ we introduce the corresponding critical Sobolev variable exponent $r^*(\cdot)$ defined by
$$
r^*(z)=\left\{
         \begin{array}{ll}
           \frac{Nr(z)}{N-r(z)}, & \hbox{ if } r(z)<N \\
           +\infty, & \hbox{ if }N\leq r(z)
         \end{array}
       \right.
\mbox{ for all } z\in\overline{\Omega}.
$$

Suppose that $r\in C^{0,1}(\overline{\Omega})\cap E_1$ and $\tau\in C(\overline{\Omega})$ with $1\leq \tau_-$. We have the following embeddings (anisotropic Sobolev embedding theorem).
\begin{prop}\label{prop1}
  \begin{itemize}
    \item[(a)] $W^{1,r(z)}(\Omega)\hookrightarrow L^{\tau(z)}(\Omega)$ continuously if $\tau(z)\leq r^*(z)$ for all $z\in\overline{\Omega}$.
    \item[(b)] $W^{1,r(z)}(\Omega)\hookrightarrow L^{\tau(z)}(\Omega)$ compactly if $\tau(z)<r^*(z)$ for all $z\in\overline{\Omega}$.
  \end{itemize}
\end{prop}

If $u\in W^{1,r(z)}(\Omega)$, then we write $\rho_r(Du)=\rho_r(|Du|)$. There is a close relation between the norm $\|\cdot\|_{r(z)}$ and the modular function $\rho_r(\cdot)$.

\begin{prop}\label{prop2}
  If $r\in E_1$ and $\{u_n,u\}_{n\in\NN}\subseteq L^{r(z)}(\Omega)$, then
\begin{itemize}
  \item[(a)] $\|u\|_{r(z)}=\vartheta \Leftrightarrow \rho_r\left(\frac{u}{\vartheta}\right)=1$.
  \item[(b)] $\|u\|_{r(z)}<1$ (resp. $=1,>1$) $\Leftrightarrow$ $\rho_r(u)<1$ (resp. $=1,>1$).
  \item[(c)] $\|u\|_{r(z)}<1$ $\Rightarrow$ $\|u\|_{r(z)}^{r_+}\leq \rho_r(u)\leq \|u\|_{r(z)}^{r_-}$.
  \item[(d)] $\|u\|_{r(z)}>1$ $\Rightarrow$ $\|u\|_{r(z)}^{r_-}\leq \rho_r(u)\leq \|u\|_{r(z)}^{r_+}$.
  \item[(e)] $\|u_n\|_{r(z)}\to0$ $\Leftrightarrow$ $\rho_r(u_n)\to 0$.
  \item[(f)] $\|u_n\|_{r(z)}\to\infty$ $\Leftrightarrow$ $\rho_r(u_n)\to+\infty$.
\end{itemize}
\end{prop}

Let $A_r: W^{1,r(z)}(\Omega)\to W^{1,r(z)}(\Omega)^*$ be the nonlinear operator defined by
$$
\langle A_r(u),h\rangle=\int_\Omega |Du|^{r(z)-2}(Du,Dh)_{\RR^N}dz
$$
for all $u,h\in W^{1,r(z)}(\Omega)$. This operator has the following properties, see Gasinski \& Papageorgiou \cite{5Gas-Pap} and R\u adulescu \& Repov\v s \cite[p. 40]{13Rad-Rep}.

\begin{prop}\label{prop3}
  The operator $A_r:W^{1,r(z)}(\Omega)\to W^{1,r(z)}(\Omega)^*$ is bounded (that is, maps bounded sets to bounded sets), continuous, monotone (thus, maximal monotone, too) and of type $(S)_+$, that is,
$$
`` u_n \overset{w}{\to}u \mbox{ in } W^{1,r(z)}(\Omega) \mbox{ and } \limsup_{n\to\infty}\langle A_r(u_n),u_n-u\rangle \leq 0
$$
$$
\mbox{ imply that }
$$
$$
u_n\to u \mbox{ in } W^{1,r(z)}(\Omega)".
$$
\end{prop}

We now recall the Weierstrass-Tonelli theorem, which we will use in the sequel. For the convenience of the reader we include also the proof.

\begin{thm}\label{weto}
If $X$ is a reflexive Banach space and $\phi:X\rightarrow \RR$ is sequentially weakly lower semicontinuous and coercive, then there exists $\hat u\in X$ such that
$$\phi(\hat u)=\inf\{\phi (u):\ u\in X\}.$$
\end{thm}

\begin{proof}
The coercivity of $\phi$ implies that for $R>0$ big we have
$$\inf_X\varphi =\inf_{\overline{B}_R}\phi,$$
with $\overline{B}_R=\{u\in X:\ \|u\|_X\leq R\}$.

On account of the reflexivity of $X$ and the Eberlein-Smulian theorem, $\overline{B}_R$ is sequentially weakly compact. Since $\varphi(\cdot)$ is sequentially weakly lower semicontinuous, we conclude that there exists $\hat u\in X$ such that
$$\phi(\hat u)=\inf_X\phi.$$
The proof is now complete.
\end{proof}

On account of the anisotropic regularity theory (see \cite[Theorem 1.3]{3Fan} and \cite[Corollary 3.1]{15Tan-Fang}), we will also use the Banach space $C^1(\overline{\Omega})$. This is an ordered Banach space with positive cone $$C_+=\{u\in C^1(\overline{\Omega}):u(z)\geq0 \mbox{ for all }z\in\overline{\Omega}\}.$$ This cone has a nonempty interior given by
$$
{\rm int}\,C_+=\{u\in C_+:\: u(z)>0 \mbox{ for all }z\in\overline{\Omega}\}.
$$
We will also use another open cone in $C^1(\overline{\Omega})$, which is defined by
$$
D_+=\left\{u\in C^1(\overline{\Omega}):\: u(z)>0 \mbox{ for all }z\in \Omega,\ \frac{\partial u}{\partial n}\Big|_{\partial\Omega\cap u^{-1}(0)}<0\right\}.
$$
Recall that $\displaystyle\frac{\partial u}{\partial n}=(Du,n)_{\RR^N}$, with $n(\cdot)$ being the outward unit normal on $\partial\Omega$.
If $h_1, h_2\in M(\Omega)$ with $h_1\leq h_2$, then we define
\begin{eqnarray*}
 &&[h_1,h_2]=\{u\in W^{1,r(z)}(\Omega):\: h_1(z)\leq u(z)\leq h_2(z)\mbox{ for a.a. }z\in\Omega\},\\
&&[h_1)=\{u\in W^{1,r(z)}(\Omega): \: h_1(z)\leq u(z) \mbox{ for a.a. }z\in\Omega\},\\
&&{\rm int}_{C^1(\Omega)}[h_1,h_2]=\mbox{ the interior in $C^1(\overline{\Omega})$ of } [h_1,h_2]\cap C^1(\overline{\Omega}).
\end{eqnarray*}
Suppose that $X$ is a Banach space and $\varphi\in C^1(X)$. We introduce the following sets
\begin{eqnarray*}
   && K_\varphi=\left\{u\in X:\: \varphi'(u)=0\right\} \mbox{ (the critical set of $\varphi$), } \\
  && \varphi^c=\{x\in X:\:\varphi(u)\leq c\} \mbox{ for all }c\in\RR.
\end{eqnarray*}

We say that $\varphi(\cdot)$ satisfies the ``$C$-condition", if it has the following property:
``Every sequence $\{u_n\}_{n\in\NN}\subseteq X$ such that $\{\varphi(u_n)\}_{n\in\NN}\subseteq\RR$ is bounded and $(1+\|u_n\|)\varphi'(u_n)\to 0$ in $X^*$ admits a strongly convergent subsequence".

This is a compactness condition of the functional $\varphi(\cdot)$ which compensates for the fact that the ambient space $X$ is not in general locally compact (being infinite dimensional). Various techniques have been proposed in the literature in order to recover the compactness in several circumstances. We refer to Tang and Cheng \cite{tangcheng} who proposed a new approach to restore the compactness of Palais-Smale sequences and to
Tang and Chen \cite{tangchen} who introduced an original method to recover the compactness of minimizing sequences.

If $Y_2\subseteq Y_1\subseteq X$, then by $H_k(Y_1,Y_2)$ (for $k\in\NN_0$), we denote the $k^{th}$ relative singular homology group with integer coefficients. If $u\in K_\varphi$ is isolated, then the $k^{th}$ critical group of $\varphi$ at $u$ is defined by
$$
C_k(\varphi,u)=H_k(\varphi^c\cap U, (\varphi^c\cap U)\setminus\{u\}) \mbox{ for all }k\in\NN_0,
$$
with $c=\varphi(u)$ and $U$ is a neighbourhood of $u$ such that $\varphi^c\cap K_\varphi\cap U=\{u\}$. The excision property of singular homology implies that this notion is well-defined, that is, independent of the choice of the isolating neighborhood $U$ (see \cite{9Pap-Rad-Rep}).

If $u\in M(\Omega)$, we set
$$
u^+(z)=\max\{u(z),0\},\ u^-(z)=\max\{-u(z),0\}, \ \mbox{for all}\ z\in \Omega.
$$
Then if $u\in W^{1,r(z)}(\Omega)$, we know that
$$
u^\pm\in W^{1,r(z)}(\Omega),\ u=u^+-u^-,\ |u|=u^+-u^-.
$$
Given a Carath\' eodory function $g:\Omega\times\RR\to\RR$, we denote by $N_g(\cdot)$  the corresponding Nemytski operator, that is,
$$
N_g(u)(\cdot)=g(\cdot,u(\cdot)) \mbox{ for all }u\in M(\Omega).
$$
Since a Carath\'eodory function is jointly measurable, $N_g(u)\in M(\Omega)$.

By $|\cdot|_N$ we denote the Lebesgue measure on $\RR^N$ and by $\|\cdot\|$ we will denote the norm of the Sobolev space $W^{1,p(z)}(\Omega)$.

Now we will introduce our hypotheses on the data of problem $(P_\lambda)$.

\smallskip
$H_0:$ $p,q\in C^{0,1}(\overline{\Omega})$ and $1<q(z)<p(z)$ for all $z\in\overline{\Omega}$.

\smallskip
$H_1:$ $f:\Omega\times\RR\to\RR$ is a Carath\' eodory function such that $f(z,0)=0$ for a.a. $z\in\Omega$ and
\begin{itemize}
  \item[(i)] $|f(z,x)|\leq a(z)\left[1+x^{r(z)-1}\right]$ for a.a. $z\in\Omega$, all $x\geq0$, with $\hat{a}\in L^\infty(\Omega)$, $r\in C(\overline{\Omega})$, $p_+<r(z)<p^*(z)$ for all $z\in\overline{\Omega}$;
  \item[(ii)] if $F(z,x)=\displaystyle{\int_0^x f(z,s)ds}$ then $\displaystyle{\lim_{x\to+\infty}\frac{F(z,x)}{x^{p_+}}}=+\infty$ uniformly for a.a. $z\in\Omega$;
  \item[(iii)] if $e(z,x)=f(z,x)x-p_+F(z,x)$, then there exists $\mu\in L^1(\Omega)$ such that
$$
e(z,x)\leq e(z,y)+\mu(z) \mbox{ for a.a. } z\in\Omega, \mbox{ all }0\leq x\leq y;
$$
  \item[(iv)] there exist $\tau\in C(\overline{\Omega})$ and $C_0,\delta_0,\hat{C}>0$ such that
\begin{eqnarray*}
   && 1<\tau_+<q_-, \\
   && C_0 x^{\tau(z)-1}\leq f(z,x) \mbox{ for a.a. }z\in\Omega, \mbox{ all }0\leq x\leq \delta_0, \\
   && -\hat{C}x^{p(z)-1}\leq f(z,x) \mbox{ for a.a. }z\in\Omega, \mbox{ all }x\geq0;
\end{eqnarray*}
  \item[(v)] for every $\rho>0$, there exists $\hat{\xi}_\rho>0$ such that for a.a. $z\in\Omega$ the function $x\mapsto f(z,x)+\hat{\xi}_\rho x^{p(z)-1}$ is nondecreasing on $[0,\rho]$.
\end{itemize}

\begin{rem}
  Since we look for positive solutions and all the above hypotheses concern the positive semiaxis $\RR_+=[0,\infty)$, without any loss of generality we assume that
\begin{equation}\label{eq1}
f(z,x)=0 \mbox{ for a.a. }z\in\Omega, \mbox{ all }x\leq 0.
\end{equation}
\end{rem}
Hypotheses $H_1(ii),(iii)$ imply that for a.a. $z\in\Omega$ the mapping $f(z,\cdot)$ is $(p_+-r)$-superlinear as $x\to+\infty$. However, this superlinearity condition is not expressed using the $AR$-condition. We recall that the $AR$-condition (unilateral version due to \eqref{eq1}) says that there exist $\vartheta>p_+$ and $M>0$ such that
\begin{eqnarray*}
   && 0<\vartheta F(z,x)\leq f(z,x)x \mbox{ for a.a. }z\in\Omega, \mbox{ all } x\geq M,\\
  && 0<\underset{\Omega}{\rm essinf}\, F(\cdot,M).
\end{eqnarray*}
These conditions imply that there exists $\tilde{C}>0$ such that
$$
\tilde{C} x^{\vartheta-1}\leq f(z,x) \mbox{ for a.a. }z\in \Omega, \mbox{ all }x\geq M.
$$
So, the $AR$-condition dictates at least $(\vartheta-1)$-polynomial growth for $f(z,\cdot)$. Hypotheses $H_1(ii),(iii)$ are less restrictive and incorporate in our framework superlinear nonlinearities with ``slower" growth as $x\to+\infty$ (see the example below). Also we emphasize that in contrast to the previous works \cite{1Den-Wan}, \cite{4Fan-Deng}, we do not assume that $f\geq0$ neither that $f(z,\cdot)$ is nondecreasing. These are hypotheses employed by Fan \& Deng \cite{4Fan-Deng} and Deng \& Wang \cite{1Den-Wan}. Moreover, in the aforementioned works the authors assume the $AR$-condition for $f(z,\cdot)$. Finally, note that also in contrast to the previous works, we do not assume that the parameter $\lambda$ is strictly positive. Here, $\lambda\in \RR$ and so the differential operator (left-hand side) of problem $(P_\lambda)$ is not in general coercive.
\begin{example}\label{exp}
  Consider the function
$$
f(z,x)=\left\{
         \begin{array}{ll}
           \vartheta(x^+)^{\tau(z)-1}-\hat{C}_0(x^+)^{\gamma(z)-1}, & \hbox{ if }x\leq 1 \\
           x^{p_+-1}\ln x+(\vartheta-\hat{C}_0)x^{\eta(z)-1}, & \hbox{ if } 1<x
         \end{array}
       \right.
$$
\end{example}
\noindent with $\tau,\gamma,\eta\in C(\overline{\Omega})$, $\tau_+<q_-$, $1<\tau(z)<\gamma(z)$, $1<\eta(z)<p(z)$ for all $z\in\overline{\Omega}$ with $\hat{C}_0>\vartheta>0$. Then this function satisfies hypotheses $H_1$ above but does not satisfy the hypotheses of \cite{1Den-Wan}, \cite{4Fan-Deng} (the $AR$-condition fails and $f(z,\cdot)$ is not monotone on $\RR_+=[0,\infty)$).

We introduce the following sets
\begin{eqnarray*}
  && \mathcal{L}^+=\left\{\lambda\in\RR: \mbox{ problem $(P_\lambda)$ has a positive solution}\right\}, \\
  && S_\lambda^+=\mbox{ the set of positive solutions of $(P_\lambda)$.}
\end{eqnarray*}
\section{Positive solutions}

We start by showing that $\mathcal{L}^+$ is nonempty. To this end, let $\eta>0$ and consider the following auxiliary anisotropic Neumann problem
\begin{equation}\label{eq2}
  \left\{
\begin{array}{lll}
-\Delta_{p(z)} u-\Delta_{q(z)} u+ u^{p(z)-1}= \eta \text{ in } \Omega,\\
\displaystyle\frac{\partial u}{\partial n}=0 \mbox{ on }\partial\Omega,\ u\geq 0.
\end{array}
\right.
\end{equation}

\begin{prop}\label{prop4}
  If hypotheses $H_0$ hold, then problem \eqref{eq2} has a unique positive solution
$$
\overline{u}_\eta\in{\rm int}\,C_+
$$
and  $\overline{u}_\eta\to 0$ in $C^1(\overline{\Omega})$ as $\eta\to 0^+$.
\end{prop}

\begin{proof}
Let $K_p:L^{p(z)}(\Omega)\to L^{p'(z)}(\Omega)$ be the nonlinear operator defined by
$$
K_p(u)(\cdot)=|u(\cdot)|^{p(\cdot)-2}u(\cdot) \mbox{ for all }u\in L^{p(z)}(\Omega).
$$
This operator is continuous and strictly monotone, too (see \cite[p.117]{9Pap-Rad-Rep}). Then we introduce $V:W^{1,p(z)}(\Omega)\to W^{1,p(z)}(\Omega)^*$ defined by
$$
V(u)=A_p(u)+A_q(u)+K_p(u) \mbox{ for all } u\in W^{1,p(z)}(\Omega).
$$
Then $V(\cdot)$ is maximal monotone (see \cite[p.135]{9Pap-Rad-Rep}), strictly monotone and
\begin{eqnarray*}
  && \langle V(u),u\rangle\geq \rho_p(Du)+\rho_p(u) \mbox{ for all }u\in W^{1,p(z)}(\Omega),\\
  &\Rightarrow& V(\cdot) \mbox{ is coercive } \mbox{ (see Proposition \ref{prop2}).}
\end{eqnarray*}
Then Corollary 2.8.7 of \cite[p.135]{9Pap-Rad-Rep} implies that $V(\cdot)$ is surjective. So, we can find $\overline{u}_\eta\in W^{1,p(z)}(\Omega)$ such that
$$
V(\overline{u}_\eta)=\eta.
$$

On account of the strict monotonicity of $V(\cdot)$, this solution is unique. Taking duality brackets with $-\overline{u}^-_\eta\in W^{1,p(z)}(\Omega)$, we obtain
\begin{eqnarray*}
  && \rho_p(D\overline{u}_\eta^-)+\rho_p(\overline{u}_\eta^-)\leq \int_\Omega \eta(-\overline{u}_\eta^-)dz\leq0, \\
  &\Rightarrow& \overline{u}_\eta\geq 0,\ \overline{u}_\eta\not=0 \mbox{ (see Proposition \ref{prop2} and recall that $\eta>0$). }
\end{eqnarray*}
We have
\begin{equation}\label{eq3}
  -\Delta_{p(z)} \overline{u}_\eta -\Delta_{q(z)} \overline{u}_\eta+ \overline{u}_\eta^{p(z)-1}=\eta \mbox{ in }\Omega,\;\frac{\partial \overline{u}_\eta}{\partial n}=0 \mbox{ on }\partial\Omega.
\end{equation}

From Winkert \& Zacher \cite{16Win-Zach} (see also Papageorgiou, R\u adulescu \& Zhang \cite[Proposition A1]{11Pap-Rad-Zhang}), we have $\overline{u}_\eta\in L^\infty(\Omega)$. Then the anisotropic regularity theory (see Fan \cite{3Fan} and Tan \& Fang \cite{15Tan-Fang}), we have $\overline{u}_\eta\in C_+\setminus\{0\}$. From \eqref{eq3} we have
\begin{eqnarray*}
   && \Delta_{p(z)} \overline{u}_\eta + \Delta_{q(z)} \overline{u}_\eta\leq \overline{u}_\eta^{p(z)-1} \mbox{ in }\Omega, \\
  &\Rightarrow& \overline{u}_\eta\in {\rm int}\,C_+\\
    && \mbox{ (see Papageorgiou, Qin \& R\u adulescu \cite[Proposition 4]{7Pap-Qin-Rad}). }
\end{eqnarray*}

Now let $\eta_n\to 0^+$ and let $\overline{u}_n=\overline{u}_{\eta_n}\in {\rm int}\,C_+$ for all $n\in \NN$. The anisotropic regularity theory (see \cite{3Fan}, \cite{15Tan-Fang}) implies that there exist $\alpha\in (0,1)$ and $C_1>0$ such that
\begin{equation}\label{eq4}
  u_n\in C^{1,\alpha}(\overline{\Omega}),\ \|u_n\|_{C^{1,\alpha}(\overline{\Omega})}\leq C_1 \mbox{ for all }n\in \NN.
\end{equation}
The compact embedding of $C^{1,\alpha}(\overline{\Omega})$ into $C^1(\overline{\Omega})$, implies that at least for a subsequence we have
\begin{equation}\label{eq5}
  \overline{u}_n\to \overline{u} \mbox{ in } C^1(\overline{\Omega}) \mbox{ as } n\to\infty.
\end{equation}
Recall that
$$
A_p(\overline{u}_n)+ A_q(\overline{u}_n)+K_p(\overline{u}_n)=\eta_n \mbox{ in $W^{1,p(z)}(\Omega$ for all } n\in\NN,
$$
that is,
$$\langle A_p(\overline{u}_n),h\rangle + \langle A_q(\overline{u}_n),h\rangle +\int_\Omega |\overline{u}_n|^{p(z)-2}\overline{u}_nhdz=\int_\Omega \eta_nhdz \mbox{ for all $h\in W^{1,p(z)}(\Omega)$.}
$$
We pass to the limit as $n\to\infty$. Because of \eqref{eq5} and Proposition \ref{prop3} we have
$$A_p(\overline{u}_n)\rightarrow A_p(\overline{u}),\ A_q(\overline{u}_n)\rightarrow A_q(\overline{u})\ \mbox{in $W^{1,p(z)}(\Omega)$}$$
and
$$\int_\Omega |\overline{u}_n|^{p(z)-2}\overline{u}_nhdz\rightarrow \int_\Omega |\overline{u}|^{p(z)-2}\overline{u}hdz.$$
Hence in the limit we have
\begin{eqnarray*}
  && A_p(\overline{u})+A_q(\overline{u})+K_p(\overline{u})=0, \\
  &\Rightarrow& \overline{u}=0.
\end{eqnarray*}
Therefore we obtain
$$
\overline{u}_\eta\to 0 \mbox{ in } C^1(\overline{\Omega}) \mbox{ as } \eta\to 0^+ \mbox{ (see \eqref{eq5}). }
$$
The proof is complete.
\end{proof}

Using Proposition \ref{prop4}, we see that for $\eta \in (0,1)$ small we have
\begin{equation}\label{eq6}
  0\leq \overline{u}_\eta(z)\leq 1 \mbox{ for all }z\in\overline{\Omega}.
\end{equation}
For such an $\eta\in(0,1)$, let $m_\eta=\displaystyle{\min_{\overline{\Omega}}\overline{u}_\eta}>0$ (recall that $\overline{u}_\eta\in {\rm int}\,C_+$). Then let
\begin{equation}\label{eq7}
  \hat{\lambda}=\frac{\|N_f(\overline{u}_\eta)\|_\infty}{m_\eta^{p_+-1}}+1>0
\end{equation}
$$
\mbox{ (see hypothesis $H_1(i)$). }
$$
We will show that $\hat{\lambda}\in\mathcal{L}^+$ and so $\mathcal{L}^+\not=\emptyset$.

\begin{prop}\label{prop5}
  If hypotheses $H_0$, $H_1$ hold, then $\mathcal{L}^+\not=\emptyset$ and $S_\lambda^+\subseteq {\rm int}\,C_+$ for every $\lambda\in\RR$.
\end{prop}

\begin{proof}
Let $\overline{u}_\eta\in{\rm int}\,C_+$ and $\hat{\lambda}>0$ be as above. We have
\begin{eqnarray}\nonumber
  && -\Delta_{p(z)} \overline{u}_\eta-\Delta_{q(z)}\overline{u}_\eta+
\hat{\lambda}\overline{u}_\eta^{p(z)-1} \\ \nonumber
   &\geq& -\Delta_{p(z)} \overline{u}_\eta-\Delta_{q(z)}\overline{u}_\eta+
\frac{f(z,\overline{u}_\eta(z))}{m_\eta^{p_+-1}}\overline{u}^{p_+-1}_\eta+
\overline{u}_\eta^{p(z)-1} \mbox{ (see \eqref{eq6}, \eqref{eq7}) } \\ \nonumber
  &\geq& \eta+f(z,\overline{u}_\eta) \mbox{ (see Proposition \ref{prop4}) } \\
  &\geq& f(z,\overline{u}_\eta) \mbox{ in } \Omega. \label{eq8}
\end{eqnarray}
We introduce the Carath\'eodory function $\hat{f}(z,x)$ defined by
\begin{equation}\label{eq9}
  \hat{f}(z,x)=\left\{
                 \begin{array}{ll}
                   f(z,x^+), & \hbox{ if }x\leq \overline{u}_\eta(z) \\
                   f(z, \overline{u}_\eta(z)), & \hbox{ if } \overline{u}_\eta(z)<x.
                 \end{array}
               \right.
\end{equation}
We set $\hat{F}(z,x)=\displaystyle{\int_0^x \hat{f}(z,s)ds}$ and consider the $C^1$-functional $\hat{\varphi}: W^{1,p(z)}(\Omega)\to \RR$ defined by
\begin{eqnarray*}
  \hat{\varphi}(u)&=& \int_\Omega \frac{1}{p(z)}|Du|^{p(z)}dz+\int_\Omega \frac{1}{q(z)}|Du|^{q(z)}dz +\hat{\lambda}\int_\Omega \frac{1}{p(z)}|u|^{p(z)}dz \\
  &-& \int_\Omega \hat{F}(z,u)dz \mbox{ for all } u\in W^{1,p(z)}(\Omega).
\end{eqnarray*}
We have
\begin{eqnarray*}
   && \hat{\varphi}(u)\geq\frac{1}{p_+}\left[\rho_p(Du)+\hat{\lambda}\rho_p(u)\right]
-\int_\Omega \hat{F}(z,u)dz,\\
  &\Rightarrow& \hat{\varphi}(\cdot) \mbox{ is coercive (see \eqref{eq9} and Proposition \ref{prop2}). }
\end{eqnarray*}
Also, using Proposition \ref{prop1} (the anisotropic Sobolev embedding theorem), we see that $\hat{\varphi}(\cdot)$ is sequentially weakly lower semicontinuous. So, by the Weierstrass-Tonelli theorem, we can find $\hat{u}\in W^{1,p(z)}(\Omega)$ such that
\begin{equation}\label{eq10}
  \hat{\varphi}(\hat{u})=\inf\left\{\hat{\varphi}(u):\:u\in W^{1,p(z)}(\Omega)\right\}.
\end{equation}
Let $u\in {\rm int}\,C_+$ and choose $t\in(0,1)$ small such that
\begin{equation}\label{eq11}
  0<tu(z)\leq \min\{\delta_0,\min_{\overline{\Omega}}\overline{u}_\eta\} \mbox{ for all }z\in \overline{\Omega}.
\end{equation}
Here $\delta_0>0$ is as in hypothesis $H_1(iv)$ and recall that $\overline{u}_\eta\in {\rm int}\,C_+$, so that $\displaystyle{\min_{\overline{\Omega}}\overline{u}_\eta>0}$. We have
\begin{eqnarray*}
  \hat{\varphi}(tu) &\leq& \frac{t^{q_-}}{q_-}\left[\rho_p(Du)+\rho_q(Du)+\hat{\lambda}\rho_p(u)\right]-
\frac{C_0 t^{\tau_+}}{\tau_+}\rho_{\tau}(u) \\
  && \mbox{ (see \eqref{eq11}, hypothesis $H_1(iv)$ and recall that $t\in(0,1)$) } \\
  &=& C_2 t^{q_-}-C_3t^{\tau_+} \mbox{ for some } C_2,C_3>0.
\end{eqnarray*}
Recall that $\tau_+<q_-$ (see hypothesis $H_1(iv)$). So, choosing $t\in(0,1)$ even smaller if necessary, we have
\begin{eqnarray*}
  && \hat{\varphi}(tu)<0, \\
  &\Rightarrow& \hat{\varphi}(\hat{u})<0=\hat{\varphi}(0) \mbox{ (see \eqref{eq10}), } \\
  &\Rightarrow& \hat{u}\not=0.
\end{eqnarray*}
From \eqref{eq10} we have
$$
\hat{\varphi}'(\hat{u})=0,
$$
\begin{equation}\label{eq12}
 \Rightarrow \langle A_p(\hat{u}),h\rangle+\langle A_q(\hat{u}),h\rangle+\hat{\lambda}\int_\Omega |\hat{u}|^{p(z)-2}\hat{u}hdz =\int_\Omega \hat{f}(z,\hat{u})hdz
\end{equation}
for all $h\in W^{1,p(z)}(\Omega)$.
In \eqref{eq12} first we choose $h=-\hat{u}\in W^{1,p(z)}$. Then we have
\begin{eqnarray*}
  && \rho_p(D\hat{u}^-)+\rho_q(D\hat{u}^-)+\hat{\lambda}\rho_p(\hat{u}^-)=0 \mbox{ (see \eqref{eq9}, \eqref{eq1}), } \\
  &\Rightarrow& \hat{u}\geq 0,\ \hat{u}\not=0.
\end{eqnarray*}
Next, in \eqref{eq12} we choose $h=(\hat{u}-\overline{u}_\eta)^+\in W^{1,p(z)}(\Omega)$. We have
\begin{eqnarray*}
   && \langle A_p(\hat{u}),(\hat{u}-\overline{u}_\eta)^+\rangle+ \langle A_q(\hat{u}), (\hat{u}-\overline{u}_\eta)^+\rangle +\hat{\lambda}\int_\Omega \hat{u}^{p(z)-1}(\hat{u}-\overline{u}_\eta)^+dz  \\
  &=& \int_\Omega f(z,\overline{u}_\eta)(\hat{u}-\overline{u}_\eta)^+dz \mbox{ (see \eqref{eq9}) } \\
  &\leq& \langle A_q(\overline{u}_\eta),(\hat{u}-\overline{u}_\eta)^+\rangle+\langle A_q(\overline{u}_\eta), (\hat{u}-\overline{u}_\eta)^+\rangle+\hat{\lambda}\int_\Omega \overline{u}_\eta^{p(z)-1}(\hat{u}-\overline{u}_\eta)^+dz \\
    && \mbox{ (see \eqref{eq8}), }\\
  \Rightarrow &&\hat{u}\leq  \overline{u}_\eta \mbox{ (see Proposition \ref{prop3}). }
\end{eqnarray*}
So, we have proved that
\begin{equation}\label{eq13}
  \hat{u}\in[0,\overline{u}_\eta],\ \hat{u}\not=0.
\end{equation}
From \eqref{eq12}, \eqref{eq13} and \eqref{eq9} it follows that
$$
\hat{u}\in S_\lambda^+ \mbox{ and so }\hat{\lambda}\in \mathcal{L}^+\not=\emptyset.
$$
Moreover, as before, the anisotropic regularity theory (see \cite{3Fan}, \cite{15Tan-Fang}) and the anisotropic maximum principle (see \cite{7Pap-Qin-Rad}), imply that
$$
S_\lambda^+\subseteq {\rm int}\,C_+ \mbox{ for all }\lambda\in \RR.
$$
The proof is now complete.
\end{proof}

Next, we show that $\mathcal{L}^+$ is connected, more precisely $\mathcal{L}^+$ is an upper half line.

\begin{prop}\label{prop6}
  If hypotheses $H_0$, $H_1$ hold, $\lambda\in \mathcal{L}$ and $\lambda<\mu<\infty$, then $\mu\in \mathcal{L}^+$.
\end{prop}

\begin{proof}

Since by hypothesis $\lambda\in \mathcal{L}^+$, we can find $u_\lambda\in S^+_\lambda\subseteq {\rm int}\,C_+$. Then we have
\begin{eqnarray}\nonumber
  && -\Delta_{p(z)} u_\lambda-\Delta_{q(z)}u_\lambda +\mu u_\lambda^{p(z)-1} \\
  &\geq& -\Delta_{p(z)}u_\lambda-\Delta_{q(z)}u_\lambda +\lambda u_\lambda^{p(z)-1}=f(z,u_\lambda) \mbox{ in }\Omega. \label{eq14}
\end{eqnarray}
Let $\vartheta>-\mu$ and consider the Carath\' eodory function $k(z,x)$ defined by
\begin{equation}\label{eq15}
  k(z,x)=\left\{
           \begin{array}{ll}
             f(z,x^+)+\vartheta(x^+)^{p(z)-1}, & \hbox{ if } x\leq u_\lambda(z) \\
             f(z,u_\lambda(z))+\vartheta u_\lambda(z)^{p(z)-1}, & \hbox{ if } u_\lambda(z)<x.
           \end{array}
         \right.
\end{equation}
Let $K(z,x)=\displaystyle{\int_0^x k(z,s)ds}$ and consider the $C^1$-functional $\hat{\varphi}_\mu: W^{1,p(z)}(\Omega)\to\RR$ defined by
\begin{eqnarray*}
  \hat{\varphi}_{\mu}(u) &=& \int_\Omega \frac{1}{p(z)}|Du|^{p(z)}dz+\int_\Omega \frac{1}{q(z)}|Du|^{q(z)}dz +\int_\Omega \frac{\vartheta+\mu}{p(z)}|u|^{p(z)}dz \\
  &-& \int_\Omega K(z,u)dz \mbox{ for all }u\in W^{1,p(z)}(\Omega).
\end{eqnarray*}
Since $\vartheta+\mu>0$ from \eqref{eq15} and Proposition \ref{prop2}, we see that
$$
\hat{\varphi}_\mu(\cdot) \mbox{ is coercive. }
$$
Also using Proposition \ref{prop1} (the anisotropic Sobolev embedding theorem), we infer that
$$
\hat{\varphi}_\mu(\cdot) \mbox{ is sequentially weakly lower semicontinuous. }
$$
So, we can find $u_\mu\in W^{1,p(z)}(\Omega)$ such that
\begin{equation}\label{eq16}
  \hat{\varphi}_\mu(u_\mu)=\inf\left\{\hat{\varphi}_\mu(u):\:u\in W^{1,p(z)}(\Omega)\right\}.
\end{equation}
As in the proof of Proposition \ref{prop5}, via hypothesis $H_1(iv)$, we show that
\begin{eqnarray*}
   && \hat{\varphi}_\mu(u_\mu)<0=\hat{\varphi}_\mu(0), \\
  &\Rightarrow& u_\mu\not=0.
\end{eqnarray*}
From \eqref{eq16} we have that
\begin{equation}\label{eq17}
  \langle\hat{\varphi}'_\mu(u_\mu),h\rangle=0 \mbox{ for all }h\in W^{1,p(z)}(\Omega).
\end{equation}
As before (see the proof of Proposition \ref{prop5}), choosing $h=-u_\mu^-\in W^{1,p(z)}(\Omega)$ and $h=(u_\mu-u_\lambda)^+\in W^{1,p(z)}(\Omega)$ in \eqref{eq17}, we show that
\begin{equation}\label{eq18}
  u_\mu\in[0,\mu_\lambda],\;u_\mu\not=0,\;u_\mu\not=u_\lambda \mbox{ (since $\lambda<\mu$). }
\end{equation}
From \eqref{eq17}, \eqref{eq18} and \eqref{eq15}, we deduce that
$$
u_\mu\in S_\mu^+\subseteq{\rm int}\,C_+ \mbox{ and so }\mu\in \mathcal{L}^+.
$$
This completes the proof.
\end{proof}

An interesting byproduct of the above proof, is the following corollary.

\begin{cor}\label{cor7}
  If hypotheses $H_0$, $H_1$ hold, $\lambda\in\mathcal{L}^+$, $u_\lambda\in S_\lambda^+$ and $\lambda<\mu<\infty$, then $\mu\in \mathcal{L}^+$ and we can find $u_\mu\in S_\mu^+$ such that
$$
u_\mu\leq u_\lambda.
$$
\end{cor}

We can improve the assertion of this corollary as follows.
\begin{prop}\label{prop8}
  If hypotheses $H_0$, $H_1$ hold, $\lambda\in \mathcal{L}^+$, $u_\lambda\in S_\lambda^+$ and $\lambda<\mu<\infty$, then $\mu\in \mathcal{L}^+$ and we can find $u_\mu\in S_\mu^+$ such that
$$
u_\lambda-u_\mu\in D_+.
$$
\end{prop}

\begin{proof}
  From Corollary \ref{cor7} we already know that $\mu\in \mathcal{L}^+$ and we can find $u_\mu\in S_\mu^+\subseteq{\rm int}\,C_+$ such that
\begin{equation}\label{eq19}
  0\leq u_\mu\leq u_\lambda.
\end{equation}

Let $\rho=\|u_\lambda\|_\infty$ and let $\hat{\xi}_\rho>0$ be as postulated by hypothesis $H_1(v)$. We have
\begin{eqnarray}\nonumber
  && -\Delta_{p(z)}u_\lambda-\Delta_{q(z)}u_\lambda +[\lambda+\hat{\xi}_\rho]u_\lambda^{p(z)-1}\\ \nonumber
  &=& f(z,u_\lambda)+\hat{\xi}_\rho u_\lambda^{p(z)-1} \\ \nonumber
  &\geq& f(z,u_\mu)+\hat{\xi}_\rho u_\mu^{p(z)-1} \mbox{ (see \eqref{eq19} and hypothesis $H_1(v)$) } \\ \nonumber
  &=& -\Delta_{p(z)} u_\mu-\Delta_{q(z)} u_\mu +\left[\mu+\hat{\xi}_\rho\right]u_\mu^{p(z)-1}  \mbox{ (since $u_\mu\in S_\mu^+$) }  \\ \nonumber
  &=& -\Delta_{p(z)} u_\mu-\Delta_{q(z)} u_\mu +[\lambda+\hat{\xi}_\rho]u_\mu^{p(z)-1}+(\mu-\lambda)u_\mu^{p(z)-1} \\
  &\geq& -\Delta_{p(z)} u_\mu-\Delta_{q(z)} u_\mu +[\lambda+\hat{\xi}_\rho]u_\mu^{p(z)-1} \mbox{ (since $\lambda<\mu$). }\label{eq20}
\end{eqnarray}
We know that $u_\mu\in {\rm int}\,C_+$. Hence $0<m_\mu=\displaystyle{\min_{\overline{\Omega}}u_\mu}$. We set $\hat{m}_\mu=\displaystyle{\min\{m_\mu,1\}}>0$. Then
$$
0<[\mu-\lambda]\hat{m}_\mu^{p_+-1}\leq [\mu-\lambda]u_\mu^{p(z)-1} \mbox{ for all }z\in \overline{\Omega.}
$$
Then from \eqref{eq20} and Proposition 5 of Papageorgiou, Qin \& R\u adulescu \cite{7Pap-Qin-Rad}, we infer that $u_\lambda-u_\mu\in D_+$.
The proof is now complete.
\end{proof}

Let $\lambda_*=\inf\mathcal{L}^+$.

\begin{prop}\label{prop9}
  If hypotheses $H_0$, $H_1$ hold, then $\lambda_*>-\infty$.
\end{prop}

\begin{proof}
  Let $\lambda>\lambda_*$. Then on account of Proposition \ref{prop6}, we have $\lambda\in\mathcal{L}^+$. So, we can find $u\in S_\lambda^+\subseteq{\rm int}\,C_+$ and we have
\begin{equation}\label{eq21}
  \langle A_p(u),h\rangle +\langle A_q(u),h\rangle+\lambda\int_\Omega u^{p(z)-1}hdz=\int_\Omega f(z,u)hdz
\end{equation}
for all $h\in W^{1,p(z)}(\Omega)$.

In \eqref{eq21} we choose $h\equiv1\in W^{1,p(z)}(\Omega)$. Then
\begin{eqnarray*}
  && \lambda\int_\Omega u^{p(z)-1}dz=\int_\Omega f(z,u)dz\geq -\hat{C}\int_\Omega u^{p(z)-1}dz \\
  &\Rightarrow& (\lambda+\hat{C})\int_\Omega u^{p(z)-1}dz\geq0\\
  &\Rightarrow& \lambda+\hat{C}\geq0 \mbox{ and so }\lambda\geq -\hat{C}.
\end{eqnarray*}
So, we conclude that $\lambda_*\geq-\hat{C}>-\infty$.
\end{proof}

By imposing a sign  condition on $f(z,\cdot)$, we can have that $\mathcal{L}^+\subseteq\RR_+=[0,+\infty]$, that is, $\lambda_*\geq0$.

The new conditions on the reaction $f(z,x)$ are the following.

\smallskip
$H'_1:$ $f:\Omega\times\RR\to\RR$ is a Carath\' eodory function such that $f(z,0)=0$ for a.a. $z\in\Omega$, hypotheses $H'_1(i),(ii),(iii),(v)$ are the same as the corresponding hypotheses $H_1(i),(ii),(iii),(v)$ and\\
$(iv)$ there exist $\tau\in C(\overline{\Omega})$ and $C_0,\delta_0>0$ such that
\begin{eqnarray*}
  && 1<\tau_+<q_- \\
  && C_0x^{\tau(z)-1}\leq f(z,x) \mbox{ for a.a. }z\in\Omega, \mbox{ all }0\leq x\leq \delta_0 \\
  \mbox{ and }&& 0\leq f(z,x) \mbox{ for a.a. }z\in\Omega, \mbox{ all }x\geq0.
\end{eqnarray*}

\begin{rem}
  So, the new conditions of $f(z,\cdot)$ require that $f(z,\cdot)\Big|_{\RR_+}$ is nonnegative (it can not change sign). This was the case with the reactions in the works of Fan \& Deng \cite{4Fan-Deng} and Deng \& Wang \cite{1Den-Wan}.
\end{rem}

Under the above stronger conditions on the reaction $f(z,\cdot)$ we can show that the set $\mathcal{L^+}$ of admissible parameters is a subset of $\RR_+$.

\begin{prop}\label{prop10}
  If hypotheses $H_0$, $H'_1$ hold, then $\lambda_*\geq0$.
\end{prop}

\begin{proof}
  Let $\lambda>\lambda_*$. We know that $\lambda\in \mathcal{L}^+$ and so there exists $u\in S_\lambda^+\subseteq{\rm int}\,C_+$. From \eqref{eq21} with $h\equiv1\in W^{1,p(z)}(\Omega)$, we have
\begin{eqnarray*}
   && \lambda\int_\Omega u^{p(z)-1}dz=\int_\Omega f(z,u)dz\geq0 \mbox{ (see $H'_1(iv)$), } \\
  &\Rightarrow& \lambda\geq0 \mbox{ and so }\lambda_*\geq0.
\end{eqnarray*}
The proof is complete.
\end{proof}

On account of hypotheses $H_1(i),(iv)$, we see that we can find $C_4>0$ such that
\begin{equation}\label{eq22}
  f(z,x)\geq C_0x^{\tau(z)-1}-C_4 x^{r(z)-1} \mbox{ for a.a. }z\in\Omega, \mbox{ all }x\geq 0.
\end{equation}
Let $\eta>0$ and let $\hat{\lambda}_\eta=\lambda_*+\eta$. Evidently, $\hat{\lambda}_\eta\in \mathcal{L}^+$ (see Corollary \ref{cor7}). The unilateral growth condition in \eqref{eq22} leads to the following auxiliary anisotropic Neumann problem
\begin{equation}\label{eq23}
\left\{
\begin{array}{lll}
-\Delta_{p(z)} u-\Delta_{q(z)} u+\hat{\lambda}_\eta |u|^{p(z)-1}= C_0 u^{\tau(z)-1}- C_4u^{r(z)-1} \text{ in } \Omega,\\
\displaystyle\frac{\partial u}{\partial n}=0 \mbox{ on }\partial\Omega,\ u\geq0.
\end{array}
\right.
\end{equation}

\begin{prop}\label{prop11}
  If hypotheses $H_0$ hold, then problem \eqref{eq23} has a unique positive solution $u_\eta^*\in{\rm int}\,C_+$.
\end{prop}

\begin{proof}
  Let $\lambda\in (\lambda_*,\hat{\lambda}_\eta]$. We know that $\lambda\in \mathcal{L}^+$ (see Proposition \ref{prop6}) and so we can find $u\in S_\lambda^+\subseteq{\rm int}\,C_+$. Let $\vartheta>-\hat{\lambda}_\eta$ and consider the Carath\' eodory function
\begin{equation}\label{eq24}
  \beta(z,x)=\left\{
               \begin{array}{ll}
                 C_0(x^+)^{\tau(z)-1}-C_4(x^+)^{r(z)-1}+\vartheta (x^+)^{p(z)-1}, & \hbox{ if  }x\leq u(z) \\
                 C_0 u(z)^{\tau(z)-1}-C_4u(z)^{r(z)-1}+\vartheta u(z)^{p(z)-1}, & \hbox{ if } u(z)<x.
               \end{array}
             \right.
\end{equation}
We set $B(z,x)=\displaystyle{\int_0^x\beta(z,s)ds}$ and consider the $C^1$-functional $\Psi:W^{1,p(z)}(\Omega)\to \RR$ defined by
\begin{eqnarray*}
  \Psi(u) &=& \int_\Omega \frac{1}{p(z)} |Du|^{p(z)}dz+\int_\Omega \frac{1}{q(z)}|Du|^{q(z)}dz +\int_\Omega \frac{\vartheta+\hat{\lambda}_\eta}{p(z)}|u|^{p(z)}dz \\
  &-& \int_\Omega B(z,u)dz \mbox{ for all }u\in W^{1,p(z)}(\Omega).
\end{eqnarray*}
From \eqref{eq24} and since $\vartheta>-\hat{\lambda}_\eta$ we see that $\Psi(\cdot)$ is coercive. Also it is sequentially weakly lower semicontinuous. So, by the Weierstrass-Tonelli theorem, we can find $u_\eta^*\in W^{1,p(z)}(\Omega)$ such that
\begin{equation}\label{eq25}
  \Psi(u_\eta^*)=\inf\left\{\Psi(u):\:u\in W^{1,p(z)}(\Omega)\right\}.
\end{equation}
Since $\tau_+<q_-<p_+<r_-$, if $v\in{\rm int}\,C_+$ and $t\in(0,1)$ is small (at least so that $tv(z)\leq u(z)$ for all $z\in\overline{\Omega}$), then
\begin{eqnarray*}
  && \Psi(tv)<0, \\
  &\Rightarrow& \Psi(u_\eta^*)<0=\Psi(0) \mbox{ (see \eqref{eq25}), } \\
  &\Rightarrow& u_\eta^*\not=0.
\end{eqnarray*}
From \eqref{eq25} we have
$$
\Psi'(u_\eta^*)=0,
$$
\begin{eqnarray}\nonumber
   \Rightarrow \langle A_p(u_\eta^*),h\rangle+\langle A_q(u_\eta^*),h\rangle\!\!\!&+&\!\!\! (\vartheta+\hat{\lambda}_\eta)\int_\Omega |u_\eta^*|^{p(z)-2}u_\eta^* hdz  \\
  &=& \int_\Omega \beta(z,u_\eta^*)hdz \mbox{ for all }h\in W^{1,p(z)}(\Omega). \label{eq26}
\end{eqnarray}
In \eqref{eq26} first we choose $h=-(u_\eta^*)^-\in W^{1,p(z)}(\Omega)$. Using \eqref{eq24} we obtain
\begin{eqnarray*}
  && \rho_p(D(u_\eta^*)^-) +\rho_q((u_\eta^*)^-) +[\vartheta+\hat{\lambda}_\eta]\rho_p((u_\eta^*)^-)=0, \\
  &\Rightarrow& u_\eta^*\geq0,\ u_\eta^*\not=0 \mbox{ (recall that $\vartheta>-\hat{\lambda}_\eta$).}
\end{eqnarray*}
Next, in \eqref{eq26} we choose $(u_\eta^*-u)^+\in W^{1,p(z)}(\Omega)$. Then we have
\begin{eqnarray*}
  && \langle A_p(u_\eta^*),(u_\eta^*-u)^+\rangle+\langle A_q(u_\eta^*),(u_\eta^*-u)^+\rangle +(\vartheta+\hat{\lambda}_\eta)\int_\Omega |u_\eta^*|^{p(z)-2}u_\eta^*hdz \\
  &=& \int_\Omega \left[C_0 u^{\tau(z)-1}-C_4u^{r(z)-1}+\vartheta u^{p(z)-1}\right](u_\eta^*-u)^+ \mbox{ (see \eqref{eq24}) } \\
  &\leq& \int_\Omega \left[f(z,u)+\vartheta u^{p(z)-1}\right](u_\eta^*-u)^+ dz \mbox{ (see \eqref{eq22}) } \\
  &\leq& \langle A_p(u),(u_\eta^*-u)^+\rangle+\langle A_q(u), (u_\eta^*-u)^+\rangle +(\vartheta+\hat{\lambda}_\eta)\int_\Omega u^{p(z)-1}(u_\eta^*-u)^+dz \\
  && \mbox{ (since $u\in S_\lambda$ and $\lambda\leq \hat{\lambda}_\eta$), }\\
  \Rightarrow&& u_\eta^*\leq u\mbox{ (see Proposition \ref{prop3}). }
\end{eqnarray*}
So, we have proved that
\begin{equation}\label{eq27}
  u_\eta^*\in[0,u],\; u_\eta^*\not=0.
\end{equation}
From \eqref{eq26}, \eqref{eq27} and \eqref{eq24} it follows that
$$
u_\eta^* \mbox{ is a positive solution of problem \eqref{eq23}. }
$$
As before, the anisotropic regularity theory (\cite{3Fan}, \cite{15Tan-Fang}) and the  anisotropic maximum principle (see \cite{7Pap-Qin-Rad}), imply that
$$
u_\eta^*\in{\rm int}\,C_+.
$$

Next, we show that this positive solution of \eqref{eq23} is unique. To this end, we consider the integral functional $j:L^1(\Omega)\to\overline{\RR}=\RR\cup\{+\infty\}$ defined by
$$
j(u)=\left\{
       \begin{array}{ll}
       \displaystyle{  \int_\Omega \frac{1}{p(z)}\left|Du^{1/q_-}\right|^{p(z)}dz+\int_\Omega \frac{1}{q(z)}\left|Du^{1/q_-}\right|^{q(z)}dz},\\
        \left(\hbox{ if } u\geq0,\ u^{1/q_-}\in W^{1,p(z)}(\Omega)\right).\\
         +\infty,  \left(\hbox{ otherwise }\right).
       \end{array}
     \right.
$$
From Theorem 2.2 of Taka\v c \& Giacomoni \cite{14Tak-Gia}, we know that $j(\cdot)$ is convex. Let ${\rm dom}\,j=\left\{u\in L^1(\Omega):\:j(u)<\infty\right\}$ (the effective domain of $j(\cdot)$). Suppose $y_\eta^*$ is another positive solution of problem \eqref{eq23}. Again we have $y_\eta^*\in{\rm int}\,C_+$. On account of Proposition 4.1.22 of Papageorgiou, R\u adulescu \& Repov\v s \cite[p. 274]{9Pap-Rad-Rep}, we have
\begin{equation}\label{eq28}
  \frac{u_\eta^*}{y_\eta^*}\in L^\infty(\Omega) \mbox{ and }\frac{y_\eta^*}{u_\eta^*}\in L^\infty(\Omega).
\end{equation}
Let $h=(u_\eta^*)^{q_-}(y_\eta^*)^{q_-}$. From \eqref{eq28} and for $|t|<1$ small, we have
$$
(u_\eta^*)+th\in{\rm dom}\:j,\;(y_\eta^*)^{q_-}+th\in {\rm dom}\:j.
$$
Exploiting the convexity of $j(\cdot)$ and using the chain rule, we see that $j(\cdot)$ is G\^ateaux differentiable at $(u_\eta^*)^{q_-}$ and at $(y_\eta^*)^{q_-}$ in the direction $h$. Moreover, via Green's identity, we have
\begin{eqnarray*}
  j'\left((u_\eta^*)^{q_-}\right)(h) &=& \frac{1}{q_-}\int_\Omega \frac{-\Delta_{p(z)}u_\eta^*-\Delta_{q(z)}u_\eta^*}{(u_\eta^*)^{q_- -1}}hdz \\
  &=& \frac{1}{q_-}\int_\Omega \left[\frac{C_0}{(u_\eta^*)^{q_- -\tau(z)}}-C_4(u_\eta^*)^{r(z)-q_-}-\hat{\lambda}_\eta(u_\eta^*)^{p(z)-q_-}
\right]hdz,
\end{eqnarray*}

\begin{eqnarray*}
  j'\left((y_\eta^*)^{q_-}\right)(h) &=& \frac{1}{q_-}\int_\Omega \frac{-\Delta_{p(z)}y_\eta^*-\Delta_{q(z)}y_\eta^*}{(y_\eta^*)^{q_- -1}}hdz \\
  &=& \frac{1}{q_-}\int_\Omega \left[\frac{C_0}{(y_\eta^*)^{q_- -\tau(z)}}-C_4(y_\eta^*)^{r(z)-q_-}-\hat{\lambda}_\eta(y_\eta^*)^{p(z)-q_-}
\right]hdz.
\end{eqnarray*}
The convexity of $j(\cdot)$ implies that $j'(\cdot)$ is monotone. Then
\begin{eqnarray*}
  0 &\leq& \int_\Omega C_0\big[\frac{1}{(u_\eta^*)^{q_- -\tau(z)}}-\frac{1}{(y_\eta^*)^{q_- -\tau(z)}}\big]hdz\\
    &-& \!\!\! \int_\Omega C_4\left[(u_\eta^*)^{r(z)-q_-}-(y_\eta^*)^{r(z)-q_-}\right]hdz \\
  &-& \!\!\!\hat{\lambda}_\eta\int_\Omega \left[(u_\eta^*)^{p(z)-q_-}-(y_n^*)^{p(z)-q_-}\right]hdz\\
  &\leq&0,
\end{eqnarray*}
$$
\Rightarrow u_\eta^*=y_\eta^* \mbox{ (recall that $\tau_+<q_-<p_-$). }
$$
This proves the uniqueness of the positive solution $u_\eta^*\in{\rm int}\,C_+$ of problem \eqref{eq23}.
\end{proof}
This unique positive solution of problem \eqref{eq23} provides a lower bound for the elements of $S_\lambda^+$ locally in $\lambda>\lambda_*$.

\begin{prop}\label{prop12}
  If hypotheses $H_0$, $H_1$ hold, $\eta>0$ and $\lambda\in (\lambda_*,\hat{\lambda}_\eta=\lambda_*+\eta]$, then $u_\eta^*\leq u$ for all $u\in S_\lambda^+$.
\end{prop}

\begin{proof}
  Let $u\in S_\lambda^+$, $\vartheta>-\hat{\lambda}_\eta$ and consider the Carath\' eodory function $k(z,x)$ defined by
\begin{equation}\label{eq29}
  k(z,x)=\left\{
           \begin{array}{ll}
             C_0(x^+)^{\tau(z)-1}-C_4(x^+)^{r(z)-1}+ \vartheta(x^+)^{p(z)-1}, & \hbox{ if }x\leq u(z) \\
             C_0u(z)^{\tau(z)-1}-C_4 u(z)^{r(z)-1}+\vartheta u(z)^{p(z)-1}, & \hbox{ if }u(z)<x.
           \end{array}
         \right.
\end{equation}
We set $K(z,x)=\displaystyle{\int_0^x k(z,s)ds}$ and consider the $C^1$-functional $\sigma:W^{1,p(z)}(\Omega)\to\RR$ defined by
\begin{eqnarray*}
  \sigma(u) &=& \int_\Omega\frac{1}{p(z)}|Du|^{p(z)}dz+\int_\Omega\frac{1}{q(z)}|Du|^{q(z)}dz
+\int_\Omega \frac{\vartheta+\hat{\lambda}_\eta}{p(z)}|u|^{p(z)}dz \\
  &-& \int_\Omega K(z,u)dz \mbox{ for all }u\in W^{1,p(z)}(\Omega).
\end{eqnarray*}
From \eqref{eq29} and since $\vartheta>-\hat{\lambda}_\eta$, we see that $\sigma(\cdot)$ is coercive. Also, using the anisotropic Sobolev embedding theorem (see Proposition \ref{prop1}), we see that
$\sigma(\cdot)$ is sequentially weakly lower semicontinuous. So, by the Weierstrass-Tonelli theorem, we can find $\hat{u}_\eta^*\in W^{1,p(z)}(\Omega)$ such that
\begin{equation}\label{eq30}
  \sigma(\hat{u}_\eta^*)=\min\left\{\sigma(u):\:u\in W^{1,p(z)}(\Omega)\right\}.
\end{equation}
Since $\tau_+<q_-\leq q(z)<p(z)$ for all $z\in\overline{\Omega}$, we see that if $v\in{\rm int}\,C_+$ and $t\in(0,1)$ is small (at least so that $tv\leq u$), we have
\begin{eqnarray*}
   && \sigma(tv)<0, \\
  &\Rightarrow& \sigma(\hat{u}_\eta^*)<0=\sigma(0) \mbox{ (see \eqref{eq30}), } \\
  &\Rightarrow& \hat{u}_\eta^*\not=0.
\end{eqnarray*}
From \eqref{eq30} we have $\sigma'(u_\eta^*)=0$, thus
\begin{eqnarray}\nonumber
  \langle A_p(\hat{u}_\eta^*),h\rangle+\langle A_q(\hat{u}_\eta^*),h\rangle \!\!\!&+& \!\!\!\int_\Omega(\vartheta+\hat{\lambda}_\eta)|\hat{u}_\eta^*|^{p(z)-2}
\hat{u}_\eta^* hdz \\
  &=& \int_\Omega k(z,\hat{u}_\eta^*)hdz \mbox{ for all }h\in W^{1,p(z)}(\Omega). \label{eq31}
\end{eqnarray}
In \eqref{eq31} first we choose $h=-(\hat{u}_\eta^*)^-\in W^{1,p(z)}(\Omega)$. Using \eqref{eq29}, we obtain
\begin{eqnarray*}
  && \rho_p(D(\hat{u}_\eta^*)^-)+\rho_q(D(\hat{u}_\eta^*)^-)+\int_\Omega (\vartheta+\hat{\lambda}_\eta)\left((\hat{u}_\eta^*)^-\right)^{p(z)}dz=0 \\
  &\Rightarrow& \hat{u}_\eta^*\geq0,\;\hat{u}_\eta^*\not=0 \mbox{ (recall that $\vartheta>-\hat{\lambda}_\eta$). }
\end{eqnarray*}
Next, in \eqref{eq31} we choose $h= (\hat{u}_\eta^*-u)^+\in W^{1,p(z)}(\Omega)$. We have
\begin{eqnarray*}
  && \langle A_p(\hat{u}_\eta^*), (\hat{u}_\eta^*-u)^+\rangle+\langle A_q(\hat{u}_\eta^*),(\hat{u}_\eta^*-u)^+\rangle+\int_\Omega (\vartheta+\hat{\lambda}_\eta)(\hat{u}_\eta^*)^{p(z)-1}(\hat{u}_\eta^*-u)^+dz\\
  &=& \int_\Omega \left[C_0u^{\tau(z)-1}-C_4 u^{r(z)-1}+\vartheta u^{p(z)-1}\right](\hat{u}_\eta^*-u)^+dz \mbox{ (see \eqref{eq29}) } \\
   &\leq& \int_\Omega \left[f(z,u)+\vartheta u^{p(z)-1}\right](\hat{u}_\eta^*-u)^+dz \mbox{ (see \eqref{eq22}) } \\
  &\leq& \langle A_p(u),(\hat{u}_\eta^*-u)^+\rangle+\langle A_q(u),(\hat{u}_\eta^*-u)^+\rangle+\int_\Omega (\vartheta+\hat{\lambda}_\eta)u^{p(z)-1}(\hat{u}_\eta^*-u)^+dz \\
  && \mbox{ (since $u\in S_\lambda^+$, $\lambda\leq \hat{\lambda}_\eta$), } \\
  \Rightarrow&& \hat{u}_\eta^*\leq u.
\end{eqnarray*}
So, we have proved that
\begin{equation}\label{eq32}
  \hat{u}_\eta^*\in[0,u], \; \hat{u}_\eta^*\not=0.
\end{equation}
From \eqref{eq31}, \eqref{eq32}, \eqref{eq29} and Proposition \ref{prop11}, we conclude that
\begin{eqnarray*}
  && \hat{u}_\eta^*=u_\eta^*, \\
  &\Rightarrow& u_\eta^* \leq u \mbox{ for all }u\in S_\lambda^+, \mbox{ all }\lambda\in (\lambda_*,\hat{\lambda}_\eta=\lambda_*+\eta].
\end{eqnarray*}
The proof is now complete.
\end{proof}

\begin{rem}
  This proposition reveals that if hypotheses $H'_1$ hold, then $\lambda_*>0$.
\end{rem}

Next, we show that for all $\lambda>\lambda_*$, we have at least two positive solutions.

\begin{prop}\label{prop13}
  If hypotheses $H_0$, $H_1$ hold and $\lambda>\lambda_*$, then problem $(P_\lambda)$ has at least two positive solutions $u_0$, $\hat{u}\in{\rm int}\,C_+$, $u_0\not=\hat{u}$.
\end{prop}

\begin{proof}
  Let $\eta,\mu\in(\lambda_*,\infty)$ such that $\lambda_*<\eta<\lambda<\mu$. We know that $\eta,\mu\in \mathcal{L}^+$ (see Proposition \ref{prop6}). Moreover, on account of Proposition \ref{prop8} we can find $u_\eta\in S_\eta^+$, $u_0\in S_\lambda^+$ and $u_\mu\in S_\mu^+$ such that
\begin{eqnarray}\nonumber
  && u_\eta-u_0 \in D_+ \mbox{ and } u_0-u_\mu\in D_+, \\
  &\Rightarrow& u_0\in{\rm int}_{C^1(\overline{\Omega})}[u_\mu,u_\eta]. \label{eq33}
\end{eqnarray}
Let $\vartheta>-\lambda$ and consider the Carath\' eodory functions $\hat{g}(z,x)$ and $g(z,x)$ defined by
\begin{equation}\label{eq34}
  \hat{g}(z,x)=\left\{
                 \begin{array}{ll}
                   f(z,u_\mu(z))+\vartheta u_\mu(z)^{p(z)-1}, & \\
                   f(z,x)+\vartheta x^{p(z)-1}, &  \\
                   f(z,u_\eta(z))+\vartheta u_\eta(z)^{p(z)-1}, &
                 \end{array}
               \right.
\end{equation}
and
\begin{equation}\label{eq35}
  g(z,x)=\left\{
           \begin{array}{ll}
             f(z,u_\mu(z))+\vartheta u_\mu(z)^{p(z)-1}, & \hbox{ if }x<u_\mu(z) \\
             f(z,x)+\vartheta x^{p(z)-1}, & \hbox{ if }u_\mu(z)\leq x.
           \end{array}
         \right.
\end{equation}
We set $\hat{G}(z,x)=\displaystyle{\int_0^x \hat{g}(z,s)ds}$, $G(z,x)=\displaystyle{\int_0^x g(z,s)ds}$ and consider the $C^1$-functionals $\hat{\gamma}_\lambda, \gamma_\lambda: W^{1,p(z)}(\Omega)\to \RR$ defined by

\begin{eqnarray*}
  \hat{\gamma}_\lambda(u) &=& \int_\Omega \frac{1}{p(z)}|Du|^{p(z)}dz+\int_\Omega \frac{1}{q(z)}|Du|^{q(z)}dz+\int_\Omega \frac{\vartheta+\lambda}{p(z)}|u|^{p(z)}dz \\
   &-& \int_\Omega \hat{G}(z,u)dz,
\end{eqnarray*}

\begin{eqnarray*}
  \gamma_\lambda(u) &=& \int_\Omega \frac{1}{p(z)}|Du|^{p(z)}dz+\int_\Omega \frac{1}{q(z)}|Du|^{q(z)}dz+\int_\Omega \frac{\vartheta+\lambda}{p(z)}|u|^{p(z)}dz \\
   &-& \int_\Omega G(z,u)dz \mbox{ for all }u\in W^{1,p(z)}(\Omega).
\end{eqnarray*}
Using \eqref{eq34} and \eqref{eq35}, we show easily that
\begin{equation}\label{eq36}
  K_{\hat{\gamma}_\lambda}\subseteq [u_\mu,u_\eta]\cap{\rm int}\,C_+ \mbox{ and } K_{\gamma_\lambda}\subseteq[u_\mu)\cap{\rm int}\,C_+.
\end{equation}
It is clear from \eqref{eq34} and \eqref{eq35} that
\begin{equation}\label{eq37}
  \gamma_\lambda\Big|_{[u_\mu,u_\eta]}=\hat{\gamma}_\lambda\Big|_{[u_\mu,u_\eta]}.
\end{equation}
Then from \eqref{eq36} and \eqref{eq37}, we see that we may assume that
\begin{equation}\label{eq38}
  K_{\hat{\gamma}_\lambda}=\{u_0\}.
\end{equation}
Otherwise, we already have a second positive solution for problem $(P_\lambda)$ and so we are done. From \eqref{eq34} and since $\vartheta>-\lambda$, we see that $\hat{\gamma}(\cdot)$ is coercive. Also, it is sequentially weakly lower semicontinuous. So, $\hat{\gamma}_\lambda(\cdot)$ has a global minimizer on account of \eqref{eq38} this global minimizer is $u_0$.  From \eqref{eq33} and \eqref{eq37} it follows that
\begin{eqnarray}\nonumber
  && u_0 \mbox{ is a local $C^1(\overline{\Omega})$-minimizer of $\gamma_\lambda(\cdot)$, } \\
  &\Rightarrow& u_0 \mbox{ is a local $W^{1,p(z)}(\Omega)$-minimizer of $\gamma_\lambda(\cdot)$ } \label{eq39} \\ \nonumber
  && \mbox{ (see Gasinski \& Papageorgiou \cite{5Gas-Pap}). }
\end{eqnarray}
From \eqref{eq36} and \eqref{eq35}, we see that we may assume that $K_{\gamma_\lambda}$ is finite, otherwise we already have a sequence of distinct positive smooth solutions for $(P_\lambda)$ and so we are done. Then from \eqref{eq39} and Theorem 5.7.6 of Papageorgiou, R\u adulescu \& Repov\v s \cite[p.367]{9Pap-Rad-Rep}, we see that there exists $\rho\in(0,1)$ small such that
\begin{equation}\label{eq40}
  \gamma_\lambda(u_0)<\inf\left\{\gamma_\lambda(u):\: \|u-u_0\|=\rho\right\}=m_\lambda.
\end{equation}
Note that if $u\in{\rm int}\,C_+$, then on account of hypothesis $H_1(ii)$, we have
\begin{equation}\label{eq41}
  \gamma_\lambda(tu)\to-\infty \mbox{ as }t\to+\infty.
\end{equation}

{\bf Claim:} $\gamma_\lambda(\cdot)$ satisfies the $C$-condition.

Consider a sequence $\{u_n\}_{n\in\NN}\subseteq W^{1,p(z)}(\Omega)$ such that
\begin{eqnarray}
  && |\gamma_\lambda(u_n)|\leq C_5 \mbox{ for some } C_5>0, \mbox{ all }n\in\NN \label{eq42}\\
  && (1+\|u_n\|)\gamma'_\lambda(u_n)\to 0 \mbox{ in } W^{1,p(z)}(\Omega)^* \mbox{ as }n\to\infty. \label{eq43}
\end{eqnarray}
From \eqref{eq43}, we have
\begin{eqnarray}\nonumber
  \Big| \langle A_p(u_n),h\rangle+\langle A_q(u_n),h\rangle \!\!\!&+&\!\!\! \int_\Omega (\vartheta+\lambda)|u_n|^{p(z)-2}u_n hdz-\int_\Omega g(z,u_n)hdz\Big| \\
  &\leq& \frac{\varepsilon_n\|h\|}{1+\|u_n\|} \mbox{ for all }h\in W^{1,p(z)}(\Omega) \mbox{ and with } \varepsilon_n\to0^+. \label{eq44}
\end{eqnarray}
In \eqref{eq44} we choose $h=-u_n^-\in W^{1,p(z)}(\Omega)$. Then
\begin{eqnarray}\nonumber
   && \rho_p(Du_n^-)+\rho_q(Du_n^-)+[\vartheta+\lambda]\rho_q(u_n^-)\leq C_6\|u_n^-\| \\ \nonumber
  && \mbox{ for some $C_6>0$, all $n\in \NN$ (see \eqref{eq35}), } \\
  &\Rightarrow& \{u_n^-\}_{n\in\NN}\subseteq W^{1,p(z)}(\Omega) \mbox{ is bounded } \label{eq45}\\ \nonumber
    && \mbox{ (see Proposition \ref{prop2} and recall that $\vartheta>-\lambda$). }
\end{eqnarray}
To show that $\{u_n\}_{n\in\NN}\subseteq W^{1,p(z)}(\Omega)$ is bounded, we need to show that $\{u_n^+\}_{n\in\NN}\subseteq W^{1,p(z)}(\Omega)$ is bounded (see \eqref{eq45}). Arguing by contradiction, suppose that at least for a subsequence we have
\begin{equation}\label{eq46}
  \|u_n^+\|\to\infty
\end{equation}
Let $v_n=\frac{u_n^+}{\|u_n^+\|}$ for $n\in \NN$. Then $\|v_n\|=1$, $v_n\geq0$ for all $n\in\NN$. So, we may assume that
\begin{equation}\label{eq47}
  v_n\overset{w}{\to}v \mbox{ in }W^{1,p(z)}(\Omega),\ v_n\to v \mbox{ in }L^{r(z)}(\Omega),\ v\geq0.
\end{equation}
Let $\Omega_+=\{z\in \Omega:\: v(z)>0\}$. First we assume that $|\Omega_+|_N>0$ (that is, $v\not=0$). Then we have
\begin{eqnarray}\nonumber
  && u_n^+ \to +\infty \mbox{ for a.a. }z\in\Omega_+ \\ \nonumber
  &\Rightarrow& \frac{F(z,u_n^+(z))}{u_n^+(z)^{p_+}}\to \infty \mbox{ for a.a. }z\in \Omega_+ \\ \nonumber
  && \mbox{ (see hypothesis $H_1(ii)$), } \\
  &\Rightarrow& \int_{\Omega_+}\frac{F(z,u_n^+)}{\|u_n^+\|^{p_+}}dz\to+\infty \mbox{  (by Fatou's lemma). } \label{eq48}
\end{eqnarray}
Note that hypotheses $H_1(i),(ii)$ imply that
$$
-C_7\leq F(z,x) \mbox{ for a.a. }z\in\Omega, \mbox{ all }x\geq0, \mbox{ some }C_7>0.
$$
Hence we have
\begin{eqnarray}\nonumber
  \int_\Omega \frac{F(z,u_n^+)}{\|u_n^+\|^{p_+}}dz &=& \int_{\Omega_+}\frac{F(z,u_n^+)}{\|u_n^+\|^{p_+}}dz+\int_{\Omega\setminus \Omega_+} \frac{F(z,u_n^+)}{\|u_n^+\|^{p_+}}dz \\ \nonumber
  &\geq& \int_{\Omega_+} \frac{F(z,u_n^+)}{\|u_n^+\|^{p_+}}dz-\frac{C_7|\Omega|_N}{\|u_N^+\|^{p_+}}, \\
\Rightarrow  \int_\Omega \frac{F(z,u_n^+)}{\|u_n^+\|^{p_+}}dz \!\!\!&\to& \!\!\! +\infty \mbox{ (see \eqref{eq48} and \eqref{eq46}). } \label{eq49}
\end{eqnarray}

On the other hand, from \eqref{eq35}, \eqref{eq43} and \eqref{eq45}, we can say that

\begin{eqnarray}\nonumber
\begin{split}
  &-\frac{1}{q_-}\Big[
\int_\Omega\frac{1}{\|u_n^+\|^{p_+-p(z)}}|Dv_n|^{p(z)}dz+
\int_\Omega\frac{1}{\|u_n^+\|^{p_+-q(z)}}|Dv_n|^{q(z)}dz \\ \nonumber
  & \ \ + \int_\Omega\frac{\vartheta+\lambda}{\|u_n^+\|^{p_+-p(z)}}|v_n|^{p(z)}dz+
\int_\Omega \frac{F(z,u_n^+)}{\|u_n^+\|^{p_+}}dz\leq C_8 \\ \nonumber
  & \ \ \ \mbox{ for some $C_8>0$, all $n\in \NN$, }
\end{split}
\end{eqnarray}

\begin{eqnarray}\label{eq50}
  \Rightarrow \int_\Omega \frac{F(z,u_n^+)}{\|u_n^+\|^{p_+}}dz\leq C_9 \mbox{ for some $C_9>0$, all $n\in \NN$.}
\end{eqnarray}
Comparing \eqref{eq50} and \eqref{eq48}, we have a contradiction.

Next, we assume that $|\Omega_+|_N=0$, (that is, $v\equiv 0$). For all $n\in\NN$, let $t_n\in[0,1]$ be  such that
\begin{equation}\label{eq51}
  \gamma_\lambda(t_nu_n^+)=\max\left\{\gamma_\lambda(tu_n^+):\:0\leq t\leq 1\right\}.
\end{equation}
For $\xi>1$, let $y_n=\xi^{1/p_-}v_n$ for all $n\in\NN$. Then
$$
y_n\overset{w}{\to}0 \mbox{ in } W^{1,p(z)}(\Omega).
$$
If follows that
\begin{equation}\label{eq52}
  \int_\Omega \frac{\vartheta+\lambda}{p(z)}y_n^{p(z)}dz\to0, \int_\Omega G(z,y_n)dz\to0.
\end{equation}

On account of \eqref{eq46}, we can find $n_0\in\NN$ such that
\begin{equation}\label{eq53}
  \frac{\xi^{1/p_-}}{\|u_n^+\|}\in(0,1] \mbox{ for all }n\geq n_0.
\end{equation}
Then for $n\geq n_0$, we have
\begin{eqnarray} \nonumber
  \gamma_\lambda(t_n u_n^+) &=& \gamma_\lambda(y_n) \mbox{ (see \eqref{eq53}, \eqref{eq51}) } \\ \nonumber
  &=& \int_\Omega \frac{1}{p(z)}\xi^{p(z)/p_-}|Dv_n|^{p(z)}dz \\ \nonumber
  &+& \int_\Omega \frac{\vartheta+\lambda}{p(z)}\xi^{p(z)/p_-}v_n^{p(z)}dz-\int_\Omega G(z,y_n)dz \\ \nonumber
  &\geq& \frac{\xi}{p_+}\left[\rho_p(Dv_n)+(\vartheta+\lambda)\rho_p(v_n)\right]-
\int_\Omega G(z,y_n)dz \\ \nonumber
  && \mbox{ (recall that $\xi>1$), } \\
  \Rightarrow&& \gamma_\lambda(t_nu_n^+)\geq\frac{C_{10}\xi}{2p_+} \label{eq54} \\ \nonumber
  && \mbox{ for some $C_{10}>0$, all $n\geq n_1\geq n_0$ (see \eqref{eq52}). }
\end{eqnarray}
Since $\xi>1$ is arbitrary, from \eqref{eq54} we infer that
\begin{equation}\label{eq55}
  \gamma_\lambda(t_nu_n^+)\to+\infty \mbox{ as }n\to\infty.
\end{equation}
We have
\begin{eqnarray}
  && \gamma_\lambda(0)=0 \mbox{ and } \gamma_\lambda(u_n^+)\leq C_{11}  \label{eq56} \\ \nonumber
  && \mbox{ for some $C_{11}>0$, all $n\in\NN$ (see \eqref{eq35}, \eqref{eq42}).}
\end{eqnarray}
From \eqref{eq55} and \eqref{eq56}, it follows that we can find $n_2\in \NN$ such that
\begin{eqnarray}\nonumber
  && t_n\in(0,1) \mbox{ for all }n\geq n_2, \\ \nonumber
  &\Rightarrow& \frac{d}{dt}\gamma_\lambda(tu_n^+)\Big|_{t=t_1}=0 \mbox{ (see \eqref{eq51}), } \\
  &\Rightarrow& \langle \gamma'_\lambda(t_n u_n^+),t_n u_n^+\rangle=0 \mbox{ for all }n\geq n_2 \label{eq57} \\ \nonumber
  &&\mbox{ (by the chain rule) }
\end{eqnarray}
Then for $n\geq n_2$ we have
\begin{eqnarray}\nonumber
  && \gamma_\lambda(t_n u_n^+) \\ \nonumber
  &=& \gamma_\lambda(t_n u_n^+)- \frac{1}{p_+}\langle \gamma'_\lambda(t_n u_n^+), t_n u_n^+\rangle \mbox{ (see \eqref{eq57}) } \\ \nonumber
  &\leq& \int_\Omega \left[\frac{1}{p(z)}-\frac{1}{p_+}\right] |Du_n^+|^{p(z)}dz + \int_\Omega \left[\frac{1}{q(z)}-\frac{1}{p_+}\right]|Du_n^+|^{q(z)}dz \\ \nonumber
  &+& \frac{1}{p_+}\int_\Omega \left[g(z,t_nu_n^+)(t_n u_n^+)-p_+ G(z,t_n u_n^+)\right]dz \\ \nonumber
  && \mbox{ (since $t_n\in(0,1)$) } \\ \nonumber
  &\leq& \int_\Omega\left[\frac{1}{p(z)}-\frac{1}{p_+}\right]|Du_n^+|^{p(z)}+
\int_\Omega \left[\frac{1}{q(z)}-\frac{1}{p_+}\right] |Du_n^+|^{q(z)}dz \\ \nonumber
  &+& \frac{1}{p_+}\int_\Omega \left[g(z,u_n^+)u_n^+-p_+ G(z,u_n^+)\right]dz +C_{12} \\ \nonumber
  && \mbox{ for some $C_{12}>0$ (see $H_1(iii)$ and \eqref{eq35}) } \\ \nonumber
  &=& \gamma_\lambda(u_n^+)-\frac{1}{p_+}\langle \gamma'_\lambda(u_n^+),u_n^+\rangle +C_{12} \\
  &\leq& C_{13} \mbox{ for some $C_{13}>0$, all $n\in \NN$ (see \eqref{eq42}, \eqref{eq43}, \eqref{eq45}). } \label{eq58}
\end{eqnarray}
Comparing \eqref{eq58} and \eqref{eq55}, we have a contradiction. Therefore $\{u_n^+\}_{n\in\NN}\subseteq W^{1,p(z)}(\Omega)$ is bounded, hence
$$
\{u_n\}_{n\in \NN} \subseteq W^{1,p(z)}(\Omega) \mbox{ is bounded (see \eqref{eq45}). }
$$
We may assume that
\begin{equation}\label{eq59}
  u_n\overset{w}{\to}u \mbox{ in }W^{1,p(z)}(\Omega) \mbox{ and } u_n\to u \mbox{ in } L^{r(z)}(\Omega).
\end{equation}
In \eqref{eq44} we choose $h=u_n-u\in W^{1,p(z)}(\Omega)$, pass to the limit as $n\to\infty$ and use \eqref{eq59}. Then
\begin{eqnarray*}
  && \lim_{n\to\infty}\left[\langle A_p(u_n),u_n-u\rangle+\langle A_q(u_n), u_n-u\rangle\right]=0, \\
  &\Rightarrow& \limsup_{n\to\infty}\left[\langle A_p(u_n),u_n-u\rangle+\langle A_q(u),u_n-u\rangle\right]\leq 0 \\
  && \mbox{ (since $A_q(\cdot)$ is monotone), } \\
  &\Rightarrow& \limsup_{n\to\infty} \langle A_p(u_n),u_n-u\rangle \leq 0 \mbox{ (see \eqref{eq59}), } \\
  &\Rightarrow& u_n\to u \mbox{ in } W^{1,p(z)}(\Omega) \mbox{ (see Proposition \ref{prop3}). }
\end{eqnarray*}
Therefore $\gamma_\lambda(\cdot)$ satisfies the $C$-condition and we have proved the Claim.
Then \eqref{eq40}, \eqref{eq41} and the Claim permit the use of the mountain pass theorem. So, we can find $\hat{u}\in W^{1,p(z)}(\Omega)$ such that
\begin{equation}\label{eq60}
  \left\{
    \begin{array}{ll}
      \hat{u}\in K_{\gamma_\lambda}\subseteq[u_\mu)\cap{\rm int}\,C_+ \mbox{ (see \eqref{eq36}) } & \\
       m_\lambda\leq \gamma_\lambda(\hat{u}) \mbox{ (see \eqref{eq40}). }&
    \end{array}
  \right.
\end{equation}
From \eqref{eq60}, \eqref{eq35}, \eqref{eq40}, we infer that
$$
\hat{u}\in S_\lambda^+\subseteq {\rm int}\,C_+ \mbox{ and }\hat{u}\not=u_0\ (\mbox{for}\ \lambda\in (\lambda_*,+\infty)).
$$
This completes the proof.
\end{proof}

We have to determine what happens with the critical parameter $\lambda_*$. We show that $\lambda_*$ is also admissible and so $\mathcal{L}^+=[\lambda_*,+\infty)$.

\begin{prop}\label{prop14}
  If hypotheses $H_0$, $H'_1$ hold, then $\lambda_*\in \mathcal{L}^+$.
\end{prop}

\begin{proof}
  Let $\{\lambda_n\}_{n\in\NN}\subseteq \mathcal{L}^+$ be such that $\lambda_n \downarrow \lambda_*$. For each $n\in\NN$, let $\eta\in (\lambda_*,\lambda_n)$. From Proposition \ref{prop6} we know that $\eta \in \mathcal{L}^+$ and so we can find $u_\eta\in S_\eta^+\subseteq{\rm int}\,C_+$. From Corollary \ref{cor7}, we know that we can find $u_n=u_{\lambda_n}\in S_{\lambda_n}^+\subseteq{\rm int}\,C_+$ with $u_\eta-u_n\in D_+$.

Consider the energy functional $\varphi_\lambda: W^{1,p(z)}(\Omega)\to\RR$ for problem $(P_\lambda)$ defined by
\begin{eqnarray*}
  \varphi_\lambda(u) &=& \int_\Omega \frac{1}{p(z)}|Du|^{p(z)}dz+\int_\Omega \frac{1}{q(z)}|Du|^{q(z)}dz+\int_\Omega \frac{\lambda}{p(z)}|u|^{p(z)}dz \\
  &-& \int_\Omega F(z,u)dz \mbox{ for all }u\in W^{1,p(z)}(\Omega).
\end{eqnarray*}
We know that $\varphi_\lambda\in C^1\left(W^{1,p(z)}(\Omega)\right)$ and from the first part of this proof (see also the proof of Proposition \ref{prop6}), we have
$$
\varphi_{\lambda_n}(u_n)<\varphi_{\lambda_n}(0)=0 \mbox{ for all }n\in \NN,
$$
\begin{eqnarray}\nonumber
  &\Rightarrow& \int_\Omega \frac{p_+}{p(z)}|Du_n|^{p(z)}dz+\int_\Omega \frac{p_+}{q(z)}|Du_n|^{q(z)}dz +\int_\Omega \frac{\lambda_n p_+}{p(z)}|u_n|^{p(z)}dz \\
  &-& \int_\Omega p_+F(z,u_n)dz\leq0 \mbox{ for all }n\in \NN. \label{eq61}
\end{eqnarray}

On the other hand, since $u_n\in S_n^+$ for all $n\in \NN$, we have
$$
\varphi'_{\lambda_n}(u_n)=0 \mbox{ for all }n\in \NN,
$$
\begin{eqnarray}
   &\Rightarrow& \rho_p(Du_n)+\rho_q(Du_n)+\lambda_n \rho_p(u_n)=\int_\Omega f(z,u_n)u_n dz \label{eq62} \\ \nonumber
  && \mbox{ for all }n\in \NN.
\end{eqnarray}
From \eqref{eq61} and \eqref{eq62}, as in the proof of Proposition \ref{prop13} (see the Claim), via a contradiction argument, we show that
$$
\{u_n\}_{n\in\NN} \subseteq W^{1,p(z)}(\Omega) \mbox{ is bounded. }
$$
We may assume that
\begin{equation}\label{eq63}
  u_n\overset{w}{\to}u_* \mbox{ in }W^{1,p(z)}(\Omega),\; u_n\to u_* \mbox{ in }L^{r(z)}(\Omega).
\end{equation}
We know that
$$
\langle \varphi'_{\lambda_n}(u_n),h\rangle=0 \mbox{ for all }h\in W^{1,p(z)}(\Omega), \mbox{ all }n\in\NN.
$$
Choosing $h=u_n-u_*\in W^{1,p(z)}(\Omega)$, passing to the limit  as $n\to \infty$ and using \eqref{eq63}, as before (see the proof of Proposition \ref{prop3}), exploiting the $(S)_+$-property of $A_p(\cdot)$ (see Proposition \ref{prop3}), we obtain
\begin{eqnarray}\nonumber
  && \limsup_{n\to\infty} \langle A_{p(z)}(u_n),u_n-u_*\rangle\leq 0,\\
  &\Rightarrow& u_n\to u_* \mbox{ in } W^{1,p(z)}(\Omega). \label{eq64}
\end{eqnarray}
Let $\mu>\lambda_n$ for all $n\in \NN$. Then $\mu\in \mathcal{L}^+$ and using Proposition \ref{prop12}, we can find $u_\mu^*\in{\rm int}\,C_+$ such that
\begin{eqnarray}\nonumber
  && u_\mu^*\leq u_n \mbox{ for all }n\in \NN, \\
  &\Rightarrow& u_\mu^*\leq u_*. \label{eq65}
\end{eqnarray}
From \eqref{eq64} it follows that
\begin{eqnarray*}
  && \langle \varphi'_\lambda(u_*),h\rangle=0 \mbox{ for all }h\in W^{1,p(z)}(\Omega), \\
  &\Rightarrow& u_*\in S_\lambda^+ \mbox{ and so }\lambda_*\in \mathcal{L}^+ \mbox{ (see \eqref{eq65}). }
\end{eqnarray*}
The proof is now complete.
\end{proof}

So, we have proved that
$$
\mathcal{L}=[\lambda_*,\infty).
$$
Summarizing, we can state the following global (with respect to the parameter $\lambda\in\RR$) multiplicity theorem for problem $(P_\lambda)$ (a bifurcation-type theorem).
\begin{thm}\label{th15}
  If hypotheses $H_0$, $H_1$ hold, then there exists $\lambda_*\in\RR$ such that
\begin{itemize}
  \item[(a)] for all $\lambda>\lambda_*$, problem $(P_\lambda)$ has at least two positive solutions $u_0,\hat{u}\in{\rm int}\,C_+$, $u_0\not=\hat{u}$;
  \item[(b)] for $\lambda=\lambda_*$, problem $(P_\lambda)$ has at least one positive solution $u_*\in{\rm int}\,C_+$;
  \item[(c)] if $\lambda<\lambda_*$, problem $(P_\lambda)$ has no positive solutions.
\end{itemize}
\end{thm}

\section{Minimal positive solution}
In this section we show that for every $\lambda\in \mathcal{L}^+=[\lambda_*,\infty)$, problem $(P_\lambda)$ has a smallest positive solution.
\begin{prop}\label{prop16}
If hypotheses $H_0$, $H_1$ hold and $\lambda\in \mathcal{L}^+=[\lambda_*,+\infty)$, then  problem $(P_\lambda)$ has a smallest positive solution $\hat{u}_\lambda^*$ (that is, $\hat{u}_\lambda^*\in S_\lambda^+$ and $\hat{u}_\lambda^*\leq u$ for all $u\in S_\lambda^+$).
\end{prop}
\begin{proof}
  From Papageorgiou, R\u adulescu \& Repov\v s \cite[Proposition 7]{10Pap-Rad-Rep}  we know that $S_\lambda^+$ is downward directed (that is, if $u_1,u_2\in S_\lambda^+$, then we can find $u\in S_\lambda^+$ such that $u\leq u_1$, $u\leq u_2$). Invoking Lemma 3.10 of Hu \& Papageorgiou \cite[p. 178]{6Hu-Pap}, we can find $\{u_n\}_{n\in \NN}\subseteq S_\lambda^+$ decreasing such that
$$
\inf S_\lambda^+=\inf_{n\in\NN} u_n.
$$
We have that
\begin{equation}\label{eq65'}
  \langle \varphi'_\lambda(u_n),h\rangle=0 \mbox{ for all }h\in W^{1,p(z)}(\Omega), \mbox{ all }n\in\NN.
\end{equation}
Choosing $h=u_n\in W^{1,p(z)}(\Omega)$, we obtain
\begin{eqnarray*}
  \rho_p(Du_n)+\rho_q(Du_n)+\lambda\rho_p(u_n) &=& \int_\Omega f(z,u_n)u_n dz \\
  &\leq& \int_\Omega |f(z,u_n)| u_1dz \\
  && \mbox{ (since $0\leq u_n\leq u_1$ for all $n\in \NN$) } \\
  &\leq& C_{14}\mbox{ for some } C_{14}>0, \mbox{ all }n\in\NN \\
  && \mbox{ (see hypothesis $H_1(i)$) } \\
  \Rightarrow\{u_n\}_{n\in\NN}\subseteq W^{1,p(z)}(\Omega)\!\!\! &&\!\!\!\!\!\!\! \mbox{  is bounded (see Proposition \ref{prop2}). }
\end{eqnarray*}
We may assume that
\begin{equation}\label{eq66}
  u_n\overset{w}{\to}\hat{u}_\lambda^* \mbox{ in } W^{1,p(z)}(\Omega),\ u_n\to \hat{u}_\lambda^* \mbox{ in } L^{r(z)}(\Omega).
\end{equation}
Choosing $h=u_n-\hat{u}_\lambda^*\in W^{1,p(z)}(\Omega)$ in \eqref{eq65'}, passing to the limit as $n\to\infty$ and using \eqref{eq66} and the $(S)_+$-property of $A_p(\cdot)$ we obtain
\begin{eqnarray} \nonumber
  && u_n\to \hat{u}_\lambda^* \mbox{ in } W^{1,p(z)}(\Omega), \\
  &\Rightarrow& \langle \varphi'(\hat{u}_\lambda^*),h\rangle=0 \mbox{ for all }h\in W^{1,p(z)}(\Omega) \mbox{ (see \eqref{eq65'}). } \label{eq67}
\end{eqnarray}
Also, if $\mu>\lambda$, then we have
\begin{equation}\label{eq68}
  u_\mu^*\leq \hat{u}_\lambda^* \mbox{ (see Proposition \ref{prop12}). }
\end{equation}
From \eqref{eq67} and \eqref{eq68}, we infer that
$$
\hat{u}_\lambda^*\in S_\lambda^+\subseteq{\rm int}\,C_+,\ \hat{u}_\lambda^*={\rm inf}S_\lambda^+.
$$
This completes the proof.
\end{proof}

\section{Nodal solutions}
In this section we prove the existence of a nodal solution (sign-changing solution) for problem $(P_\lambda)$.

If the conditions on $f(z,\cdot)$ are bilateral (that is, they are valid for all $x\in \RR$ and not only for $x\geq0$ as in $H_1$), then we can have similar results for the negative solutions of $(P_\lambda)$. So, now we impose the following conditions of $f(z,\cdot)$:
\smallskip
$H''_1$: $f:\Omega\times\RR\to\RR$ is a Carath\' eodory function such that $f(z,0)=0$ for a.a. $z\in \Omega$ and

\begin{itemize}
  \item[(i)] $|f(z,x)|\leq a(z)\left[1+|x|^{r(z)-1}\right]$ for a.a. $z\in\Omega$, all $x\in \RR$ with $a\in L^\infty(\Omega)$, $r\in C(\overline{\Omega})$, $p_+<r(z)<p^*(z)$ for all $z\in \overline{\Omega}$;
  \item[(ii)] if $F(z,x)=\displaystyle{\int_0^x f(z,s)ds}$, then $\displaystyle{\lim_{x\to\pm\infty}\frac{F(z,x)}{x^{p_+}}=+\infty}$ uniformly for a.a. $z\in\Omega$;
  \item[(iii)] if $e(z,x)=f(z,x)x-p_+F(z,x)$, then there exists $\mu\in L^1(\Omega)$ such that
$$
e(z,x)\leq e(z,y)+\mu(z)
$$
for a.a. $z\in\Omega$, all $0\leq x \leq y$ and $y\leq x\leq 0$;
  \item[(iv)] there exist $\tau\in C(\overline{\Omega})$ and $C_0,\delta_0,\hat{C}>0$ such
\begin{eqnarray*}
  && 1<\tau_+<q_-, \\
  && C_0|x|^{\tau(z)}\leq f(z,x)x \mbox{ for a.a. $z\in\Omega$, all $|x|\leq \delta_0$, } \\
  && -\hat{C}|x|^{p(z)}\leq f(z,x)x \mbox{ for a.a. $z\in \Omega$, all $x\in \RR$; }
\end{eqnarray*}
  \item[(v)] for every $\rho>0$, there exists $\hat{\xi}_\rho>0$ such that for a.a. $z\in\Omega$ the function
$$
x\mapsto f(z,x)+\hat{\xi}_\rho|x|^{p(z)-2}x
$$
is nondecreasing on $[-\rho,\rho]$.
\end{itemize}

Let $\mathcal{L}^-$ be the set of admissible parameters for negative solutions and let $S_\lambda^-$ be the set of negative solutions. Then as in Section 3, we can establish the existence of a critical parameter value $\lambda^*>-\infty$ such that
$$
\mathcal{L}^-=[\lambda^*,+\infty) \mbox{ and } \emptyset\not=S_\lambda^-\subseteq-{\rm int}\,C_+ \mbox{ for all }\lambda\in \mathcal{L}^-.
$$

We have a global multiplicity result for negative solutions (see Theorem \ref{th15}). Moreover, for every $\lambda\in \mathcal{L}^-=[\lambda_*,+\infty)$ there exists a maximal negative solution $\hat{v}_\lambda^*\in S_\lambda^-\subseteq{\rm int}\,C_+$ (that is, $\hat{v}_\lambda^*\leq v$ for all $v\in S_\lambda^-$).

We set $\tilde{\lambda}_0=\max\{\lambda_*,\lambda^*\}$. For every $\lambda\geq \tilde{\lambda}_0$ the problem has extremal constant sign solutions
$$
\hat{u}_\lambda^*\in S_\lambda^*\subseteq{\rm int}\,C_+,\ \hat{v}_\lambda^*\subseteq-{\rm int}\,C_+.
$$
Let $\lambda\geq\tilde{\lambda}_0$ and $\vartheta>-\lambda$. We introduce the Carath\' eodory function $\hat{k}(z,x)$ defined by
\begin{equation}\label{eq69}
  \hat{k}(z,x)=\left\{
                 \begin{array}{ll}
                   f(z,\hat{v}_\lambda^*(z))+
\vartheta|\hat{v}_\lambda^*(z)|^{p(z)-2}\hat{v}_\lambda^*(z), & \hbox{ if } x<\hat{v}_\lambda^*(z) \\
                   f(z,x)+\vartheta|x|^{p(z)-2}x, & \hbox{ if } \hat{v}_\lambda^*(z)\leq x\leq \hat{u}_\lambda^*(z) \\
                   f(z,\hat{u}_\lambda^*(z))+
\vartheta\hat{u}_\lambda^*(z)^{p(z)-1}, & \hbox{ if } \hat{u}_\lambda^*(z)<x.
                 \end{array}
               \right.
\end{equation}
We also consider the positive and negative truncations of $\hat{k}(z,\cdot)$, namely the Carath\' eodory functions
\begin{equation}\label{eq70}
  \hat{k}_\pm(z,x)=\hat{k}(z,\pm x^{\pm}).
\end{equation}
We set
$$
\hat{K}(z,x)=\int_0^x \hat{k}(z,s)ds \mbox{ and } \hat{K}_\pm(z,x)=\int_0^x \hat{k}_\pm(z,s)ds
$$
and introduce the $C^1$-functionals $\hat{w}_\lambda,\hat{w}_\lambda^\pm: W^{1,p(z)}(\Omega)\to\RR$ defined by
\begin{eqnarray*}
  \hat{w}_\lambda(u) &=& \int_\Omega \frac{1}{p(z)}|Du|^{p(z)}dz+\int_\Omega \frac{1}{q(z)}|Du|^{q(z)}dz+\int_\Omega \frac{\vartheta+\lambda}{p(z)}|u|^{p(z)}dz \\
  &-& \int_\Omega \hat{K}(z,u)dz \\
  \hat{w}_\lambda^\pm(u) &=& \int_\Omega \frac{1}{p(z)}|Du|^{p(z)}dz+\int_\Omega \frac{1}{q(z)}|Du|^{q(z)}dz+\int_\Omega \frac{\vartheta+\lambda}{p(z)}|u|^{p(z)}dz \\
  &-& \int_\Omega \hat{K}_\pm(z,u)dz \mbox{ for all }u\in W^{1,p(z)}(\Omega).
\end{eqnarray*}
Using \eqref{eq69} and \eqref{eq70}, we can show easily that
\begin{eqnarray*}
  && K_{\hat{w}_\lambda}\subseteq[\hat{v}_\lambda^*,\hat{u}_\lambda^*]\cap C^1(\overline{\Omega}), \\
  && K_{\hat{w}_\lambda^+}\subseteq [0,\hat{u}_\lambda^*]\cap C_+, \; K_{\hat{w}_\lambda^-}\subseteq[\hat{v}_\lambda^*,0]\cap(-C_+).
\end{eqnarray*}
The extremality of $\hat{u}_\lambda^*$, $\hat{v}_\lambda^*$ implies that
\begin{equation}\label{eq71}
  K_{\hat{w}_\lambda}\subseteq[\hat{v}_\lambda^*,\hat{u}_\lambda^*]\cap C^1(\overline{\Omega}),\; K_{\hat{w}_\lambda^+}=\{0,\hat{u}_\lambda^*\},
\;K_{\hat{w}_\lambda^-}=\{0,\hat{v}_\lambda^*\}.
\end{equation}
Working with these functionals, we produce a nodal (sign-changing) solution.
\begin{prop}\label{prop17}
  If hypotheses $H_0$, $H''_1$ hold and $\lambda\geq \tilde{\lambda}_0$, then problem $(P_\lambda)$ admits a nodal solution
$$
y_0\in[\hat{v}_\lambda^*,\hat{u}_\lambda^*]\cap C^1(\overline{\Omega}).
$$
\end{prop}
\begin{proof}
  First we show that $\hat{u}_\lambda^*\in{\rm int}\,C_+$ and $\hat{v}_\lambda^*\in-{\rm int}\,C_+$ are local minimizers of the functional $\hat{w}_\lambda(\cdot)$. From \eqref{eq69} and \eqref{eq70}, we see that $\hat{w}_\lambda^+$ is coercive. Also, it is sequentially weakly lower semicontinuous. So, we can find $\overline{u}_\lambda^*\in W^{1,p(z)}(\Omega)$ such that
\begin{equation}\label{eq72}
  \hat{w}_\lambda^+(\overline{u}_\lambda^*)=\inf\left\{\hat{w}_\lambda^+(u):\:
u\in W^{1,p(z)}(\Omega)\right\}.
\end{equation}
If $u\in{\rm int}\,C_+$ and we choose $t\in (0,1)$ small so that at least we have $tu\leq \hat{u}_\lambda^*$ (recall that $\hat{u}_\lambda^*\in{\rm int}\,C_+$). Then on account of hypothesis $H''_1(iv)$ and since $\tau_+<q_-$, for $t\in (0,1)$ even smaller, we have
\begin{eqnarray*}
  && \hat{w}_\lambda^+(tu)<0, \\
  &\Rightarrow& \hat{w}_\lambda^+(\overline{u}_\lambda^*)<0=\hat{w}_\lambda^+(0) \mbox{ (see \eqref{eq72}), } \\
  &\Rightarrow& \overline{u}_\lambda^*\not=0.
\end{eqnarray*}
Since $\overline{u}_\lambda^*\in K_{w^+_\lambda}$ (see \eqref{eq72}), from \eqref{eq71} we infer that
$$
\overline{u}_\lambda^*=\hat{u}_\lambda^* \in{\rm int}\,C_+.
$$
It is clear from \eqref{eq69} and \eqref{eq70} that
\begin{eqnarray*}
  && \hat{w}_\lambda\Big|_{C_+}=\hat{w}_\lambda^+\Big|_{C_+}, \\
  &\Rightarrow& \hat{u}_\lambda^* \mbox{ is a local $C^1(\overline{\Omega})$-minimizer of $w_\lambda(\cdot)$, } \\
  &\Rightarrow& \hat{u}_\lambda^* \mbox{ is a local  $W^{1,p(z)}(\Omega)$-minimizer of $w_\lambda(\cdot)$} \\
  && \mbox{ (see Gasinski \& Papageorgiou \cite{5Gas-Pap}). }
\end{eqnarray*}

Similarly we show that $\hat{v}_\lambda^*\in-{\rm int}\,C_+$ is a local minimizer of $\hat{w}(\cdot)$. This time we work with $\hat{w}_\lambda^-(\cdot)$. We may assume that
\begin{equation}\label{eq73}
  K_{\hat{w}_\lambda} \mbox{ is finite. }
\end{equation}
Otherwise, on account of \eqref{eq71} and the extremality of $\hat{u}_\lambda^*$ and $\hat{v}_\lambda^*$, we have a whole sequence of distinct nodal solutions and so we are done. We may assume that
$$
\hat{w}_\lambda(\hat{v}_\lambda^*)\leq \hat{w}_\lambda(\hat{u}_\lambda^*).
$$
The reasoning is similar, if the opposite inequality holds. From the fact that $\hat{u}_\lambda^*$ is a local minimizer of $w_\lambda(\cdot)$, from \eqref{eq73} and by using Theorem 5.7.6 of Papageorgiou, R\u adulescu \& Repov\v s \cite[p. 449]{9Pap-Rad-Rep}, we can find $\rho\in(0,1)$ small such that
\begin{equation}\label{eq74}
  \left\{
    \begin{array}{ll}
      \hat{w}_\lambda(\hat{v}_\lambda^*)\leq \hat{w}_\lambda(\hat{u}_\lambda^*)<\inf\left\{\hat{w}(u):\: \|u-\hat{u}_\lambda^*\|=\rho\right\}=\hat{m}_\lambda, &\\
      \|\hat{v}_\lambda^*-\hat{u}_\lambda^*\|>\rho.&
    \end{array}
  \right.
\end{equation}
Evidently, the functional $\hat{w}_\lambda(\cdot)$ is coercive (see \eqref{eq69} and recall that $\vartheta>-\lambda$). So, it satisfies the $C$-condition (see Proposition 5.1.15 of \cite[p. 369]{9Pap-Rad-Rep}). Then using also \eqref{eq74}, we see that we can apply the mountain pass theorem and produce $y_0\in W^{1,p(z)}(\Omega)$ such that
\begin{eqnarray*}
  && y_0\in K_{w_\lambda}\subseteq[\hat{v}_\lambda^*,\hat{u}_\lambda^*]\cap C^1(\overline{\Omega}) \mbox{ (see \eqref{eq71}), } \\
  && \hat{w}_\lambda(\hat{v}_\lambda^*)\leq \hat{w}_\lambda(\hat{u}_\lambda^*)<\hat{m}_\lambda\leq \hat{w}_\lambda(y_0) \mbox{ (see \eqref{eq74}). }
\end{eqnarray*}
From the above we see that
$$
y_0\not\in\{\hat{u}_\lambda^*,\hat{v}_\lambda^*\}.
$$
From Theorem 6.5.8 of \cite[p. 527]{9Pap-Rad-Rep} we have
\begin{equation}\label{eq75}
  C_1(\hat{w}_\lambda, y_0)\not=0.
\end{equation}

On the other hand, hypothesis $H_1(iv)$ and Proposition 3.7 of Papageorgiou \& R\u adulescu \cite{8Pap-Rad}, imply that
\begin{equation}\label{eq76}
  C_k(\hat{w}_\lambda,0)=0 \mbox{ for all }k\in \NN_0.
\end{equation}
Comparing \eqref{eq75} and \eqref{eq76}, we conclude that
\begin{eqnarray}\nonumber
  && y_0\not=0, \\
  &\Rightarrow& y_0\not\in \{0,\hat{u}_\lambda^*,\hat{v}_\lambda^*\} \label{eq77}.
\end{eqnarray}
Since $y_0\in [\hat{v}_\lambda^*,\hat{u}_\lambda^*]\cap C^1(\overline{\Omega})$, the extremality of $\hat{u}_\lambda^*$, $\hat{v}_\lambda^*$ and \eqref{eq77} imply that $y_0\in C^1(\overline{\Omega})$ is a nodal solution of $(P_\lambda)$.
\end{proof}
So, we can state the following multiplicity theorem for our problem.

\begin{thm}\label{th18}
   If hypotheses $H_0$, $H''_1$ hold, then there exists $\tilde{\lambda}_0\in\RR$ such that
\begin{itemize}
  \item[(a)] for $\lambda=\tilde{\lambda}_0$, problem $(P_\lambda)$ has at least three nontrivial solutions
\begin{eqnarray*}
  && u_0\in{\rm int}\,C_+,\; v_0\in-{\rm int}\,C_+ \\
  && y_0\in[v_0,u_0]\cap C^1(\overline{\Omega}) \mbox{ nodal.}
\end{eqnarray*}
  \item[(b)] for all $\lambda>\tilde{\lambda}_0$, problem $(P_\lambda)$ has at least five nontrivial solutions
\begin{eqnarray*}
  && u_0,\hat{u}\in{\rm int}\,C_+,\ u_0\leq\hat{u},\ u_0\not=\hat{u}, \\
  && v_0,\hat{v}\in-{\rm int}\,C_+,\ \hat{v}\leq v_0,\ v_0\not=\hat{v}, \\
  && y_0\in[v_0,u_0]\cap C^1(\overline{\Omega}) \mbox{ nodal. }
\end{eqnarray*}
\end{itemize}
\end{thm}

\subsection*{Acknowledgments}
The authors would like to thank the referee for his/her comments and remarks.
The research was supported by the Slovenian Research
Agency program P1-0292.
The research of Vicen\c tiu D. R\u adulescu was supported by a grant of the Romanian Ministry of Research, Innovation and Digitization, CNCS/CCCDI--UEFISCDI, project number PCE 137/2021, within PNCDI III.

\end{document}